\documentclass[11pt]{amsart}

\usepackage[utf8]{inputenc}
\usepackage[english]{babel}
\usepackage[a4paper,twoside,top=1.2in, bottom=1.1in, left=0.8in, right=0.8in]{geometry}

\usepackage{graphicx}
\usepackage{caption} 
\usepackage{amsmath}
\usepackage{amsthm}
\usepackage{amssymb}
\usepackage{esint} 
\usepackage{mathrsfs} 
\usepackage{xcolor}
\usepackage{array}
\usepackage{hhline}
\usepackage{enumitem} 
\usepackage{comment} 
\usepackage[toc,page]{appendix} 
\usepackage{xparse} 
\usepackage{mathtools}
\usepackage{fancyhdr} 
\usepackage{ifthen} 
\usepackage{forloop} 
\usepackage{xstring}
\usepackage{emptypage} 
\usepackage{setspace}
\usepackage[initials,alphabetic]{amsrefs} 

\usepackage{tikz}
\usetikzlibrary{arrows,shapes,patterns,calc,fadings,decorations.pathreplacing,decorations.markings,decorations.pathmorphing,backgrounds}

\usepackage{hyperref} 
\hypersetup{
	colorlinks = true,
	linkcolor = {blue},
	urlcolor = {red},
	citecolor = {blue}
}

\usepackage[nameinlink,capitalise,noabbrev]{cleveref} 

\newcounter{results}[section] 

\theoremstyle{plain}
\newtheorem{theorem}[results]{Theorem}
\newtheorem{lemma}[results]{Lemma}
\newtheorem{claim}[results]{Claim}
\newtheorem{proposition}[results]{Proposition}
\newtheorem{corollary}[results]{Corollary}

\newtheorem*{theorem*}{Theorem}
\newtheorem*{lemma*}{Lemma}
\newtheorem*{proposition*}{Proposition}
\newtheorem*{corollary*}{Corollary}
\newtheorem*{exercise*}{Exercise}
\newtheorem*{fact*}{Fact}
\newtheorem*{conjecture*}{Conjecture}

\theoremstyle{remark}
\newtheorem{remark}[results]{Remark}

\newtheorem*{remark*}{Remark}
\newtheorem*{question*}{Question}

\theoremstyle{definition}
\newtheorem{definition}[results]{Definition}

\newtheorem*{definition*}{Definition}
\newtheorem*{example*}{Example}

\numberwithin{equation}{section}

\crefname{theorem}{Theorem}{Theorems}
\crefname{lemma}{Lemma}{Lemmas}
\crefname{claim}{Claim}{Claims}
\crefname{proposition}{Proposition}{Propositions}
\crefname{corollary}{Corollary}{Corollaries}
\crefname{definition}{Definition}{Definitions}
\crefname{example}{Example}{Examples}
\crefname{remark}{Remark}{Remarks}
\crefname{figure}{Figure}{Figures}

\newcounter{alpharesults}

\newtheorem{alphatheorem}[alpharesults]{Theorem}
\crefname{alphatheorem}{Theorem}{Theorems}

\ifdefined\comma 
        \renewcommand{\comma}{\ensuremath{\, \text{, }}}
\else 
        \newcommand{\comma}{\ensuremath{\, \text{, }}}
\fi

\newcommand{\N}{\ensuremath{\mathbb N}}
\newcommand{\R}{\ensuremath{\mathbb R}}

\DeclarePairedDelimiter\abs{\lvert}{\rvert} 
\DeclarePairedDelimiter\norm{\lVert}{\rVert} 
\newcommand{\sk}[2]{\ensuremath{\langle #1 , #2 \rangle}} 

\newcommand{\st}{\ensuremath{\ :\ }} 
\newcommand{\eqdef}{\ensuremath{\coloneqq}} 

\DeclareMathOperator{\tr}{tr}

\renewcommand{\d}{\ensuremath{\mathrm{d}}} 

\DeclareMathOperator{\II}{I\!I} 
\DeclareMathOperator{\Ric}{Ric} 

\newcommand{\Haus}{\ensuremath{\mathscr H}} 
\DeclareMathOperator{\ind}{ind} 

\newcommand{\area}{\ensuremath{\Haus^2}} 

\DeclareMathOperator{\loc}{loc}
\newcommand{\jac}{\ensuremath{J}} 

\renewcommand{\a}{\alpha}
\newcommand{\e}{\epsilon}
\renewcommand{\d}{\delta}
\renewcommand{\l}{\lambda}

\renewcommand{\phi}{\varphi}

\newcommand{\fnu}{w}
\colorlet{myGray}{gray}
\colorlet{myBlue}{blue}
\colorlet{myBlack}{black}
\colorlet{myBackground}{gray!10}

\title[Ancient solutions to FBMCF]{Ancient solutions to free boundary mean curvature flow}
\author{Theodora Bourni}

\author{Giada Franz}
\newcommand\printaddress{{
\setlength{\parindent}{17pt}
\footnotesize
\bigskip
\par 
{\scshape \noindent Theodora Bourni}
\newline 
Department of Mathematics, University of Tennessee Knoxville, Knoxville, TN 37996-1320, USA.
\newline
\textit{E-mail address:} 
\texttt{tbourni@utk.edu}
\newline 
\par\medskip
{\scshape \noindent Giada Franz}
\newline CNRS and Université Gustave Eiffel, LAMA, 77420 Champs-sur-Marne, France.
\newline
\textit{E-mail address:} 
\texttt{giada.franz@cnrs.fr}
\par
}}

\begin{document}
\begin{abstract}
We establish rigidity results for ancient solutions to the free boundary mean curvature flow in manifolds with convex boundary. In particular, we show that any free boundary minimal hypersurface of Morse index $I$ admits an $I$-parameter family of ancient solutions that emanate from it. Moreover, among ancient solutions that backward converge exponentially fast to the minimal hypersurface, these exhaust all possibilities. Additionally, we construct a smooth free boundary mean convex foliation around an unstable free boundary minimal hypersurface that enables us to provide a more detailed geometric description of mean-convex ancient solutions that backward converge to that minimal surface.
\end{abstract}
\maketitle

\section{Introduction}

Morse theory provides a powerful framework for studying the topology of a space through the analysis of differentiable functions defined on it. Remarkably, many of its central ideas extend to infinite-dimensional settings, although their implementation requires substantially more delicate analysis. Foundational contributions in this direction include \cite{Smale1961,Palais1963,Chang1993,AbbondandoloMajer2001}.
More recently, Morse-theoretic methods have played an important role in the study of geometric variational problems, and in particular of the area functional on submanifolds of a fixed ambient manifold; see for example \cite{Tromba1977,Pitts1981,White1991,MarquesNeves2014,ChenGaspar2025}.

In this paper, we focus on Morse theory for the area functional restricted to \emph{smooth} submanifolds.
The gradient flow of this functional is the mean curvature flow, which can be described analytically as a one-parameter family of submanifolds $\Sigma_t$ in a Riemannian ambient manifold $(M,g)$ evolving by
\[
\frac{\partial X}{\partial t} = \vec H(X,t)\,, \quad X \in \Sigma_t \,,
\]
where $\vec H(X,t)$ denotes the mean curvature vector of $\Sigma_t$ at $X$.
This flow arises naturally in physics as a model for evolutionary processes governed by surface tension, such as the evolution of grain boundaries in annealing metals \cite{Mullins1956}. From a mathematical point of view, mean curvature flow was first studied systematically by Brakke \cite{Brakke1978} using geometric measure theory, and later by Huisken \cite{Huisken1984} via a more classical PDE approach. Since then, the subject has developed into a very active area of research.

Here, we are interested in the case of mean curvature flow of hypersurfaces in the presence of a boundary. The natural Neumann boundary value problem for mean curvature flow, known as the \emph{free boundary problem}, prescribes that the evolving hypersurface has boundary constrained to move on a fixed barrier hypersurface and meets it orthogonally. This problem was introduced by Huisken \cite{Huisken1989} in the non-parametric setting and further developed by Stahl \cite{Stahl1996Convergence,Stahl1996Regularity}, and later by \cite{Freire2010, Wheeler2014,Edelen2016,HirschLi2023,LangfordZhu2023}.

A fundamental role in free boundary mean curvature flow is played by \emph{free boundary minimal hypersurfaces}, which arise as stationary solutions of the flow equation. Recently, these objects have been the subject of intense research activity, due to fascinating new existence results, such as \cite{FraserSchoen2016,FranzSchulz2026,KarpukhinKusnerMcGrathStern2024}.
Most of these recent examples consist of \emph{unstable} critical points of the area functional. Therefore, by analogy with classical Morse theory, one expects the existence of flow lines emanating from these free boundary minimal hypersurfaces. More precisely, one expects \emph{ancient solutions} of mean curvature flow, namely solutions that have existed for all times in the past, which backward converge to unstable free boundary minimal hypersurfaces.

Morse-theoretic considerations also suggest that the dimension of the space of ancient solutions backward converging to $\Sigma$ is equal to the Morse index of $\Sigma$, i.e., the dimension of the negative directions of the area functional at $\Sigma$ at second order.
Our first main result makes this heuristic precise and may be summarized informally as follows. 

\begin{alphatheorem}[cf.\ \cref{construction,uniqueness}, \cref{COR}] \label{thmA}
Let $\Sigma$ be a free boundary minimal hypersurface of Morse index $I$ in a Riemannian manifold $(M,g)$ with convex boundary. Then there exists an $I$-parameter family of ancient free boundary mean curvature flows emanating from $\Sigma$. Moreover, any ancient solution with sufficiently fast decay to $\Sigma$ belongs to this family.
\end{alphatheorem}

While such behavior is conjectured by analogy with classical Morse theory, the analysis of ancient solutions to free boundary mean curvature flow is far from trivial and was largely unexplored, due to the presence of the boundary interacting with the infinite dimensional nature of the problem and the potential degeneracy of the critical points. With the exception of recent classification results for convex solutions in dimension one and in highly symmetric higher-dimensional settings \cite{BourniLangford2023,BourniLangford2025,BourniBurnsCatron2025}, there has been essentially no systematic study of ancient solutions.
The present work constitutes, to the best of our knowledge, the first  classification theory for general higher-dimension ancient solutions to the free boundary mean curvature flow. 
 
Analogous classification results have been obtained in the boundaryless case by Choi--Mantoulidis \cite{ChoiMantoulidis2022} and extended to the noncompact case by Choi--Huang--Lee \cite{ChoiHuangLee2025}. Earlier results of this type have been known to be true for nonlinear parabolic PDEs in various settings (see e.g.\ \cite{Lunardi1995}).
Our results are inspired by these papers, primarily by \cite{ChoiMantoulidis2022}.
However, compared to the boundaryless case, the presence of a free boundary introduces substantial analytical difficulties, especially because the boundary condition does not lead to a linear boundary term. The underlying parabolic PDE becomes significantly more involved, the associated stability operator requires a much more delicate analysis, and the linearization procedures are considerably more intricate due to the interaction between interior geometry and boundary behavior.

We also note that deforming a free boundary minimal hypersurface by its first (or any) eigenfunction does not preserve the orthogonality condition at the boundary. The ``correct'' deformation is far from trivial and requires more delicate handling. We achieve this via the implicit function theorem and show that one can indeed deform to first order by  keeping the orthogonality at the boundary. In fact, we construct a smooth free boundary mean-convex foliation around unstable free boundary minimal hypersurfaces. We then use this foliation to obtain a different construction of mean-convex ancient solutions that also provides a more detailed geometric description. 

\begin{alphatheorem}[cf.\ \cref{MCconstruction,COR}] \label{thmB}
Let $\Sigma$ be an unstable free boundary minimal hypersurface in a Riemannian manifold $(M,g)$ with convex boundary. Then there exists a mean-convex ancient solution emanating from $\Sigma$, converging exponentially fast at a rate determined by the first eigenvalue. Moreover, if $\Sigma$ is nondegenerate, this solution is unique, up to time translation, among mean-convex ancient solutions converging to $\Sigma$ in the $C^{1,\alpha}$-topology.
\end{alphatheorem}

\subsection{Further literature}

As described above, ancient solutions naturally appear when taking a Morse-theoretic perspective on (free boundary) minimal hypersurfaces, and they provide guidance for variational problems (see e.g.\ \cite[Section~1.1]{ChuLi2024} for a discussion related to min-max theory).
Historically, they were first studied because they play a central role in the analysis of singularity formation \cite{Hamilton1994}, and they have an intrinsic geometric interest due to their strong rigidity and symmetry properties (see, for example, \cite{BourniLangfordTinaglia2022} and the references therein).

In the boundaryless setting, extensive classification results are known in the convex regime. Under assumptions such as uniform convexity, bounded eccentricity, type-I curvature decay, or bounded isoperimetric ratio, the only compact convex ancient solutions are shrinking spheres \cite{HuiskenSinestrari2015}; see also \cite{DaskalopoulosHamiltonSesum2010, HaslhoferHershkovits2016, Langford2017}. When the ambient space is the sphere, the only geodesically convex ancient solutions are shrinking hemispheres \cite{BryanLouie2016, HuiskenSinestrari2015}.

In Euclidean space, however, shrinking spheres are not the only compact convex ancient solutions. There exist families of solutions that contract to round points as $t \to 0$ but become increasingly eccentric as $t \to -\infty$ \cite{AngenentDaskalopoulosSesum2019, HaslhoferHershkovits2016, White2003}. These exhaust all compact convex ancient solutions that are non-collapsed, or equivalently entire \cite{AngenentDaskalopoulosSesum2020, BourniLangfordLynch2023, BrendleNaff2024}. In the collapsed setting, a unique rotationally symmetric example is known \cite{BourniLangfordTinaglia2021}, while without symmetry assumptions a rich family of examples already appears in dimension two \cite{BourniLangfordTinaglia2022}.

Outside the convex regime, classification becomes significantly more difficult. Under strong decay assumptions as $t \to -\infty$, ancient solutions converge backward to minimal hypersurfaces, motivating the problem of classifying ancient solutions emanating from minimal hypersurfaces from another perspective.

\subsection{Future directions}

For the uniqueness statement of \cref{thmA} (cf.\ \cref{uniqueness}), we require that the ancient solution can be expressed as a graph over a free boundary minimal hypersurface. Moreover, we assume the graph to have small parabolic $C^{1,\a}$-norm and to converge to zero sublinearly in the $C^0$-norm. When the underlying minimal hypersurface is nondegenerate, the sublinear decay assumption can be dropped.
We conjecture that the sublinear decay assumption can be dropped in the mean-convex case as well, namely there exists a unique ancient solution converging in $C^{1,\a}$ to any given unstable (free boundary) minimal hypersurface. 

A natural further problem is to classify ancient solutions in a fixed ambient manifold without assuming backward convergence to a minimal hypersurface.
In this context, it is interesting to ask whether it is possible to construct solutions whose backward limit is singular, particularly in the mean-convex setting. Similarly, one may investigate whether ancient solutions with unbounded area can be mean convex. In an extreme scenario, it would be interesting to construct ancient solutions that ``fill up'' the ambient manifold.

\subsection{Plan of the paper} In \cref{sec:definitions}, we recall basic definitions and we set the notation. In \cref{sec:GraphsOverFBMS}, we obtain several elliptical and parabolic estimates for graphs over a free boundary minimal hypersurface. In \cref{sec:construction}, we prove existence of an $I$-parameter family of ancient solutions emanating from a free boundary minimal hypersurface, by careful application of a fix point theorem.  In \cref{sec:rigidity}, we obtain uniqueness of the constructed family of ancient solutions, by use of an ODE lemma due to Merle--Zaag \cite{MerleZaag1998}, as adapted by \cite{ChoiMantoulidis2022}. Finally, in \cref{sec:meanconvex}, we provide a different more geometric proof of the existence of  mean-convex ancient solutions, by the use of barriers constructed via the implicit function theorem.

\subsection*{Acknowledgements}
We would like to thank Kyeongsu Choi and Christos Mantoulidis for clarifying certain points in their paper \cite{ChoiMantoulidis2022}, especially for pointing out the use of the ``absorption lemma'' in the proof of \cref{absorption} together with the Schauder estimates. 
We would also like to thank Lucas Ambrozio for pointing us to his paper \cite{Ambrozio2015}.

T.\,B.\ was supported by NSF grant DMS-2405007, and G.\,F.~was supported by NSF grant DMS-2405361. 
Moreover, part of this work was performed while the authors were in residence at the Simons Laufer Mathematical Sciences Institute (formerly MSRI) during the Fall 2024 semester, supported by NSF grant DMS-1928930.

\section{Definitions and notation} \label{sec:definitions}

In this paper, we consider a compact Riemannian manifold $(M^{m+1},\sk{\cdot}{\cdot})$ with boundary. We denote by $\eta$ the outward unit co-normal to $\partial M$, in such a way that the second fundamental form of $\partial M\subset M$ is given by $\II^{\partial M}(X,Y) = -\sk{\nabla_XY}{\eta}$. We assume that $M$ has convex boundary, i.e., $\II^{\partial M}>0$. Observe that $\II^{\partial M} > 0$ if for example $M$ is the unit ball in $\R^{m+1}$
(and thus $\partial M$ is the unit sphere).

We also consider a properly embedded, smooth, two-sided, hypersurface $\Sigma^m\subset M$, which we assume to be free boundary in $M$, namely $\Sigma$ intersects $\partial M$ orthogonally (or, equivalently, $\eta$ coincides with the outward unit co-normal vector field to $\partial\Sigma$, see \cref{fig:notation}). 
We denote by $\nu$ a choice of a global unit normal vector field on $\Sigma$, by $A(X,Y) = (\nabla_XY)^\perp$ the second fundamental form of $\Sigma\subset M$, and by $\vec H_\Sigma = \tr A$ its mean curvature. Note that $\vec H_\Sigma$ points in the direction where the area ``decreases''. Moreover, let $H_\Sigma= -\sk{\vec H _\Sigma}{\nu}$ denote the scalar mean curvature with respect to the choice of unit normal $\nu$. 
Recall that $\Sigma$ is a free boundary minimal surface if it has free boundary with $H_\Sigma=0$.

\subsection {The Jacobi operator} \label{sec:JacobiOp}
We define $Q$ to be the symmetric bilinear form associated to the Jacobi operator $\jac_\Sigma=\Delta +q$, where we use the notation  $q=\Ric^M(\nu, \nu)+|A|^2$, namely\footnote{In the paper, we omit the measures of integration: $d\Haus^m$ and $d\Haus^{m-1}$ for integrals on $\Sigma$ and $\partial\Sigma$ respectively, and $dt$ (or $ds$) for integrals on the real line.}
\begin{equation}\label{bilinear}
Q(u, v)=\int_\Sigma \left(\nabla u\cdot \nabla v-q uv\right)-\int_{\partial \Sigma} \II^{\partial M} (\nu,\nu)uv\,,\quad \forall u, v\in W^{1,2}(\Sigma)\,.
\end{equation}
Recall that, if $\Sigma$ is a free boundary minimal hypersurface, then $Q(u,u)$ coincides with the second variation of the area functional along a variation generated by the vector field $u\nu$.

We let $\lambda_1<\lambda_2\le \lambda_3\le \ldots\to+\infty$ be the discrete spectrum of the elliptic problem
\begin{equation} \label{eq:JacobiEigenvalueProblem}
\begin{cases}
-\jac_\Sigma\varphi = \lambda\varphi & \text{on $\Sigma$}\\
\frac{\partial \varphi}{\partial \eta} = \II^{\partial M}(\nu,\nu)\varphi &\text{on $\partial\Sigma$}\,,
\end{cases}
\end{equation}
with a choice of associated eigenfunctions $\varphi_1,\varphi_2,\varphi_3,\ldots$, forming an orthonormal basis for $L^2(\Sigma)$.
The index $\ind(\Sigma)$ of $\Sigma$ is the number of negative eigenvalues (with multiplicity) of the Jacobi operator on~$\Sigma$ and the nullity $\operatorname{nul}(\Sigma)$ of $\Sigma$ is the multiplicity of zero as eigenvalue of the Jacobi operator.

Observe that the eigenfunctions $\varphi_k$ of the Jacobi operator satisfy
\[
Q(\phi_k, \phi_k)=\int_\Sigma -\jac_\Sigma \phi_k\phi_k+\int_{\partial \Sigma} \left(\frac{\partial \phi_k}{\partial \eta}-\II^{\partial M} (\nu,\nu)\phi_k\right)\phi_k=\l_k\,.
\]

Moreover, let us recall the variational characterization of the first eigenvalue $\l_1$:
\begin{equation}\label{Ral}
\l_1=\inf_{u\in W^{1,2}(\Sigma)}\frac{\int_\Sigma(\abs{\nabla u}^2-q u^2)-\int_{\partial \Sigma}\II^{\partial M}(\nu, \nu)u^2}{\int_\Sigma u^2}=\inf_{u\in W^{1,2}(\Sigma)}\frac{Q(u,u)}{\int_\Sigma u^2}\,.
\end{equation}

\subsection{Weighted norms}

We use the following norm notation.
For any function 
\[
f\colon\Sigma\times (-\infty, 0]\to \R\,,
\]
we define its parabolic $L^1$-norm as
\[
\|f\|_{L^1(\Sigma\times (-\infty,0])}=\int_{-\infty}^0\norm{f(\cdot, t)}_{L^1(\Sigma)} \,,
\]
and its parabolic $C^{k,\a}$-norm, where $\a\in(0,1)$, as
\[
\|f\|_{C^{k,\a}(\Sigma\times (t-1,t))}=\sum_{i+2j\le k}\sup_{\Sigma\times (t-1,t)}|\partial^i_x\partial^j_tf|+\sum_{i+2j=k}[\partial^i_x\partial^j_tf]_{C^{\a}(\Sigma\times (t-1,t))}\,,
\]
where
\[
[f]_{C^{\a}(\Sigma\times (t-1,t))}=\sup_{\stackrel{(x_i, t_i)\in \Sigma\times (t-1,t)}{(x_1, t_1)\ne (x_2, t_2)}}\frac{|f(x_1, t_1)-f(x_2, t_2)|}{d(x_1, x_2)^{\a}+|t_1-t_2|^{\frac\a2}}\,.
\]

Moreover, for $\l\le 0$, we define the following weighted $L^2$-norm 
\[
\|f\|_{L^{2,\l}(\Sigma\times(-\infty,0])}=\sup_{t\le 0} \{e^{\l t}\|f(\cdot, t)\|_{L^2(\Sigma)}\}\,,
\]
and the weighted H\"older norm
\[
    \|f\|_{C^{k,\alpha,\l}(\Sigma\times(-\infty,0])}=\sup_{t\le 0} \{e^{\l t}\|f\|_{C^{k,\a}(\Sigma\times (t-1,t))}\}\,.
\]

\begin{remark}
Observe that some other papers use the subscript ``$P$'' to denote the parabolic norms, meaning $C^{k,\a}_P$ instead of $C^{k,\a}$ (see e.g.\ \cite{ChoiMantoulidis2022}). Here we omit the subscript to make the notation lighter, but we specify the domain to make clear if the norm is parabolic or it is the norm of a fixed-time slice.
Moreover, observe that the weighted norms $L^{2,\l}$ and $C^{k,\a,\l}$ always refer to parabolic norms. Therefore, if we do not specify the domain of the norms $L^{2,\l}$ and $C^{k,\a,\l}$, we mean $L^{2,\l}(\Sigma\times(-\infty,0])$ and $C^{k,\a,\l}(\Sigma\times(-\infty,0])$.
\end{remark}

\section{Graphs over a free boundary minimal hypersurface} \label{sec:GraphsOverFBMS}

\subsection {Graphs using a local parametrization}\label{sec:GraphsViaFoliation}
We let  $\bar\nu$ be an extension of the vector field $\nu$ (the choice of unit normal to $\Sigma$)  to all of $M$ such that $\bar\nu$ is tangent to $\partial M$ and let $\psi_\Sigma$ be the flow of $\overline \nu$ starting from~$\Sigma$, that is $\psi_\Sigma(x, 0)=x$ for all $x\in \Sigma$. Note that 
\[
\psi_\Sigma\colon\Sigma\times(-r_0,r_0)\to M
\] is a diffeomorphism between $\Sigma\times(-r_0,r_0)$ for some $r_0>0$ and a tubular neighborhood of $\Sigma$ in $M$, so it provides a foliation of this neighborhood.
 Using the variable notation $\psi_\Sigma(x, s)$, we have that $\frac{d}{ds}\psi_\Sigma=\bar\nu$. We remark that $\bar \nu$ is not necessarily the normal to the leaves $\psi_\Sigma (\Sigma, s)$. However, since it is for the leaf $s=0$, we have that 
\begin{equation}\label{eq:normaltofoliation}
\langle \nabla_{e_i}\psi_\Sigma, \bar\nu\rangle=\int_0^s\frac{d}{d\bar s} \langle \nabla_{e_i}\psi_\Sigma, \bar\nu\rangle =  s \sk{e_i}{\nabla_\nu\bar\nu} + \int_0^s \frac{d^2}{d\bar s^2}\sk{\nabla_{e_i}\psi_\Sigma}{\bar\nu}(s-\bar s)\,,
\end{equation}
where $\{e_i\}$, $i=1,\dots, m$, is an orthonormal frame of the tangent space of $\Sigma$. 
Similarly,
\begin{equation} \label{eq:TiltBasis}
\begin{split}
\langle \nabla_{e_i}\psi_\Sigma, \nabla_{e_j}\psi_\Sigma\rangle-\sk{e_i}{e_j} &= \int_0^s\frac{d}{d\bar s} \langle \nabla_{e_i}\psi_\Sigma, \nabla_{e_j}\psi_\Sigma \rangle \\
&= -2\sk{A(e_i,e_j)}{\nu} s + \int_0^s \frac{d^2}{d\bar s^2}\sk{\nabla_{e_i}\psi_\Sigma}{\nabla_{e_j}\psi_\Sigma} (s-\bar s)\,,
\end{split}
\end{equation}
In particular, note that both $\abs{\sk{\nabla_{e_i}\psi_\Sigma}{\bar\nu}}$ and $\abs{\sk{\nabla_{e_i}\psi_\Sigma}{\nabla_{e_j}\psi_\Sigma} - \sk{e_i}{e_j}}$ on $\Sigma\times(-r_0,r_0)$ are bounded by $C\abs{s}$, where $C>0$ is a constant that depends\footnote{Note that $\nabla_{\partial_s}\nabla_{e_i}\psi_\Sigma = \nabla_{\partial_s} \frac{\partial}{\partial x_i}\psi_\Sigma=\nabla_{e_i}\frac{\partial}{\partial s}\psi_\Sigma = \nabla_{e_i} \bar\nu$.} on $\|\bar\nu\|_{C^1(\Sigma\times(-r_0,r_0))}$ and $\|\psi_\Sigma\|_{C^1(\Sigma\times(-r_0,r_0))}$.

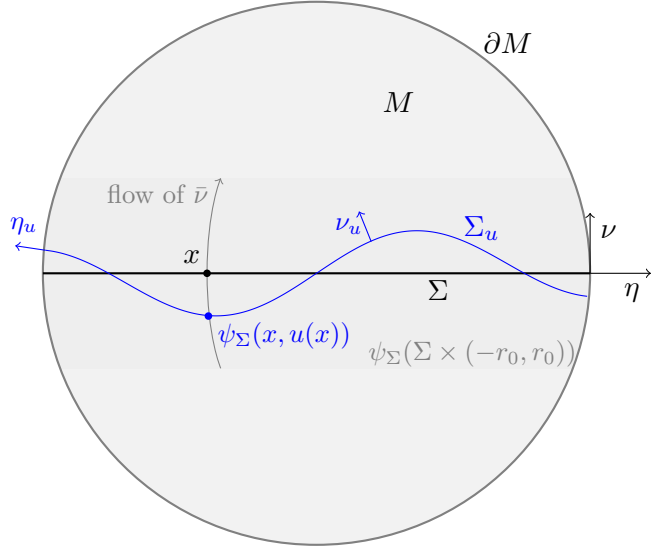
\begin{figure}[htbp]
\begin{tikzpicture}[scale=2]
\def\R{1.8}
\def\h{0.35}
\fill[gray!10] (0,0) circle (\R);
\begin{scope}
\clip (-\R,-\R*\h) rectangle (\R,\R*\h);
  \fill[gray!13] (0,0) circle (\R);
\end{scope}
\node[gray] at (\R*0.57,-\h*\R+0.1) {\small $\psi_\Sigma(\Sigma\times(-r_0,r_0))$};
\draw[gray,thick] (0,0) circle (\R);
\draw[thick] (-\R,0) -- (\R,0);
\node at (0.8,-0.11) {$\Sigma$};
\draw[blue, smooth, domain=-\R*0.99:\R*0.99, samples=120]
  plot (\x, {0.25*sin(2.5*\x r) + 0.05*\x});
\node[blue, right] at (0.9,0.3) {$\Sigma_u$};
\draw[blue,->] (-\R*0.99, {0.25*sin(2.5*(-\R*0.99) r) - 0.05*\R*0.99}) -- node[right,xshift=-4mm,yshift=3mm] {\small $\eta_u$} (-\R*0.99 - 0.2, {0.25*sin(2.5*(-\R*0.99) r) - 0.05*\R*0.99 + 0.03});
\draw[blue,->] (\R*0.2, 0.21) -- node[left,xshift=1mm,yshift=0mm] {\small $\nu_u$} (\R*0.2-0.08, 0.21 + 0.2);
\draw [gray,->] plot [smooth, tension=1.2] coordinates { (-0.35*\R,-\R*\h) (-0.4*\R,0) (-0.35*\R,\R*\h)};
\node[gray] at (-0.58*\R,\R*\h-0.1) {\small flow of $\bar\nu$};
\fill (-0.4*\R,0) circle (0.025) node [above, xshift=-2mm] {$x$};
\fill[blue] (-0.395*\R,-\R*\h*0.45) circle (0.025) node [below, xshift=10mm,yshift=0.5mm] {\small $\psi_\Sigma(x,u(x))$};
\draw[->] (\R,0) -- node[below,xshift=1.5mm] {$\eta$}(\R+0.4,0);
\draw[->] (\R,0) -- node[right,yshift=1.5mm] {$\nu$} (\R,0.4);
\node at (\R*0.3,\R*\h*1.8) {$M$};
\node at (\R*0.7,\R*0.85) {$\partial M$};
\end{tikzpicture}
\caption{Schematic representation of the notation.}
\label{fig:notation}
\end{figure}

Consider now a smooth function $u\colon\Sigma\to \R$ with $\norm{u}_{C^0(\Sigma)}< r_0$, and the hypersurface $\Sigma_u\subset M$ defined as the graph of $u$ via $\psi_\Sigma$. In other words, $\Sigma_u$ is defined as the image of the function $F\colon\Sigma\to M$, given by 
\[
F(x)\eqdef\psi_\Sigma(x, u(x))\,.
\]
Then, with $e_i$ as above, a basis for the tangent space of  $\Sigma_u$ is given by
\begin{equation} \label{eq:Fi}
F_i\eqdef \nabla _{e_i} \psi_\Sigma+u_i\overline \nu\,,\quad i=1,\dots, m\,,
\end{equation}
the metric is  $g_{ij}=\langle F_i, F_j\rangle$, and the unit normal $\nu_u$ is given by
\begin{equation} \label{eq:GraphicalNormal}
\nu_u=\frac{\bar \nu- \fnu_i g^{ij} F_j}{\sqrt{1-g^{ij}\fnu_i\fnu_j}}\,,\quad \text{where $\fnu_i\eqdef \langle F_i, \bar\nu\rangle$}  \quad \implies  \quad\sk{\nu_u}{\bar\nu} = \sqrt{1-g^{ij}w_iw_j}\,.
\end{equation}
Observe that, by \eqref{eq:normaltofoliation},
\begin{equation} \label{eq:Ri}
\fnu_i= u_i + u \sk{ e_i}{\nabla_\nu\bar\nu} + \int_0^u\frac{d^2}{ds^2}\sk{\nabla_{e_i}\psi_\Sigma}{\bar \nu} (u-s)\,.
\end{equation}
Moreover, by \eqref{eq:TiltBasis} and \eqref{eq:Fi},
\begin{equation} \label{eq:MetricGraph}
\begin{split}
g_{ij} - g_{ij}^\Sigma &= u_iu_j + \int_0^u \frac{d}{ds} \sk{\nabla_{e_i}\psi_\Sigma}{\nabla_{e_j}\psi_\Sigma}+ u_i \int_0^u \frac{d}{ds} \sk{\nabla_{e_i}\psi_\Sigma}{\bar\nu} + u_j\int_0^u \frac{d}{ds} \sk{\nabla_{e_j}\psi_\Sigma}{\bar\nu} \\
&= u_iu_j- 2\sk{A(e_i,e_j)}{\nu} u + \int_0^u\frac{d^2}{ds^2}\sk{\nabla_{e_i}\psi_\Sigma}{\nabla_{e_j}\psi_\Sigma} (u-s) +{}\\
&\phantom{=} {} + u_i \int_0^u \frac{d}{ds} \sk{\nabla_{e_i}\psi_\Sigma}{\bar\nu} + u_j\int_0^u \frac{d}{ds} \sk{\nabla_{e_j}\psi_\Sigma}{\bar\nu}\,.
\end{split}
\end{equation}
In particular, $\abs{\fnu_i-u_i}\le C\abs{u}$ and $\abs{g_{ij}-g_{ij}^\Sigma}\le C\norm{u}_{C^1}$, where $C>0$ is a constant that depends on $\Sigma$ and~$\bar\nu$.

The linearization of the mean curvature $H_u g^{ij}\sk{\nabla^2_{e_i,e_j}F}{\nu_u}$ of $\Sigma_u$ is given by $J_\Sigma u$ and the linearization of the contact angle $\langle \nu_u, \eta\rangle $ is given by $-\frac{\partial u}{\partial \eta}+\II^{\partial M}(\nu, \nu) u$ (see for example \cite{Ambrozio2015}). More precisely, in what follows (see \eqref{gmcf}), we will need the linearization of $\frac{H_u}{\langle \nu_u, \overline\nu\rangle}$, which is also $\jac_\Sigma$ since the linearization of $\sk{\nu_u}{\bar\nu}$ is $1$ (by \eqref{eq:GraphicalNormal} with \eqref{eq:Ri}, \eqref{eq:MetricGraph}). This justifies the following definition.

\begin{definition}
Given $u\colon \Sigma\to\R$ with $\norm{u}_{C^0(\Sigma)}< r_0$, we define the functions $E(u)\colon\Sigma\to\R$ and $\e(u)\colon\partial \Sigma\to\R$ as
\begin{equation}\label{linearization}
\begin{cases}
\displaystyle\frac{H_u}{\langle \nu_u, \overline\nu\rangle}=\jac_\Sigma u+E(u) & \text{on $\Sigma$}\\[3ex]
\langle \nu_u, \eta\rangle=-\frac{\partial u}{\partial \eta}+\II^{\partial M}(\nu, \nu)u+\e(u) & \text{on $\partial\Sigma$} \,.
\end{cases}
\end{equation}
\end{definition}

We need the following estimates on the error terms $E$ and $\e$. 
Observe that an estimate for $E$ in a more general case is provided in \cite{Simon1983} (see also \cite{ChoiMantoulidis2022}) and a precise expression in the case of $M=\Sigma\times\R$ is provided in \cite{ChoiHuangLee2025}. Therefore, we here focus on the boundary error term $\e$, for which we obtain a precise expression below in \eqref{e-form}.

\begin{proposition} \label{prop:TimeSliceEstimates}
For every smooth function $u\colon \Sigma\to\R$ with $\norm{u}_{C^1(\Sigma)} < r_0$, and for $E$ and $\e$ as in~\eqref{linearization}, the following estimates hold.
\begin{align*}
&\norm{E(u)}_{C^0(\Sigma)}\le C\|u\|_{C^1(\Sigma)}\|u\|_{C^2(\Sigma)}\,,
&&\norm{E(u)}_{C^\a(\Sigma)}\le C\|u\|_{C^{1,\a}(\Sigma)}\|u\|_{C^{2,\a}(\Sigma)}\,, \\
&\norm{\e(u)}_{C^0(\partial\Sigma)}\le C\|u\|^2_{C^1(\Sigma)}\,,
&&\norm{\e(u)}_{C^{1,\a}(\partial\Sigma)}\le C\|u\|_{C^{1,\a}(\Sigma)}\|u\|_{C^{2,\a}(\Sigma)}\,, \\
&\norm{E(u)}_{L^2(\Sigma)}\le C\|u\|_{C^2(\Sigma)}\|u\|_{W^{1,2}(\Sigma)}\,, 
&&\norm{\e(u)}_{L^2(\partial\Sigma)}\le C\|u\|_{C^{2}(\Sigma)}\|u\|_{W^{1,2}(\Sigma)}\,,
\end{align*}
where $C>0$ is a constant depending on $\Sigma$, $M$ and $r_0$.
\end{proposition}

\begin{remark}
The constant $r_0$ for bounding the $C^2$-norms of $u$ can be changed with any other constant, as long as we assume $\norm{u}_{C^0(\Sigma)} < r_0$ (in such a way that $\psi_\Sigma(x,u(x))$ is well-defined). We choose to bound the $C^2$-norms by $r_0$ to simplify the exposition.
\end{remark}

In order to prove the previous proposition, we use the following general lemma.

\begin{lemma} \label{lem:sigma-estimates}
Let $\tilde\sigma\colon\Sigma\times\R\times\R^m\times\R^{m\times m}\to\R$ be a $C^3$-function. 
Assume that $\tilde\sigma(x,0,0,0) = 0$, $\frac{d}{ds}|_{s=0}\tilde\sigma(x,sa,sb,sc) = 0$ for all $(a,b,c)\in\R\times\R^m\times\R^{m\times m}$, and $\tilde\sigma$ is linear in the variable $c$, namely $\tilde\sigma(x,a,b,c) = \tilde\sigma_1(x,a,b)\cdot c+\tilde\sigma_2(x,a,b)$ for two functions $\tilde\sigma_1,\tilde\sigma_2\colon \Sigma\times\R\times\R^m\to\R$.

For every function $u\colon\Sigma\to\R$ with $\norm{u}_{C^1(\Sigma)}<r_0$, defining $\sigma(u)\eqdef\tilde\sigma(x,u,\nabla u,\nabla^2 u)$, we have
\[
\norm{\sigma(u)}_{C^0(\Sigma)}\le C\norm{u}_{C^1(\Sigma)}\norm{u}_{C^2(\Sigma)} \quad \text{and} \quad \norm{\sigma(u)}_{C^\alpha(\Sigma)} \le C \norm{u}_{C^{1,\alpha}(\Sigma)} \norm{u}_{C^{2,\alpha}(\Sigma)}\,,
\]
where $C>0$ is a constant depending on $\tilde\sigma$ and $r_0$.

\end{lemma}
\begin{proof}
Denote $q(u)\eqdef (u,\nabla u,\nabla^2 u)$, in such a way that $\sigma(u) = \tilde\sigma(q(u))$ (where we are omitting the dependence on the variable $x$). Defining $f(s) = sq(u)$ for $s\in[0,1]$, by assumption we have that $(\tilde\sigma\circ f)(0)=0$ and $(\tilde\sigma\circ f)'(0)=0$. Therefore, we get
\[
\sigma(u) = (\tilde\sigma\circ f)(1) =\int_0^1(1-s) (\tilde\sigma\circ f)''(s)  = \int_0^1 (1-s)\nabla^2 \tilde\sigma(q(su))[q(u),q(u)] \,,
\]
where the hessian $\nabla^2\tilde\sigma$ is computed with respect to the variables $(a,b,c)\in\R\times\R^m\times\R^{m^2}$.
Using that $\nabla^2 \tilde\sigma(q(su))[q(u),q(u)]$ is a linear function with respect to $\nabla^2 u$ (which follows from the fact that $\tilde\sigma$ depends linearly on $c$), this implies that
\[
\norm{\sigma(u)}_{C^0(\Sigma)} \le C\norm{u}_{C^1(\Sigma)}\norm{u}_{C^2(\Sigma)}\,,
\]
where is a constant depending on $\tilde\sigma$ and $r_0$.

Similarly, using that $\norm{f\cdot g}_{C^\alpha} \le \norm{f}_{C^0}\norm{g}_{C^\alpha}+\norm{f}_{C^\alpha}\norm{g}_{C^0} \le 2\norm{f}_{C^\alpha}\norm{g}_{C^\alpha}$, we have that
\[
\norm{\sigma(u)}_{C^\alpha(\Sigma)} \le \norm{\nabla^2 \tilde\sigma(q(su))[q(u),q(u)]}_{C^\alpha(\Sigma)} \le C \norm{u}_{C^{1,\alpha}(\Sigma)} \norm{u}_{C^{2,\alpha}(\Sigma)}\,. \qedhere
\]
\end{proof}

\begin{proof}[Proof of \cref{prop:TimeSliceEstimates}]
Note that the estimates on $E$ are exactly as those derived in  \cite{Simon1983} and used in \cite{ChoiMantoulidis2022, ChoiHuangLee2025}.
A way to obtain these estimates is to observe that, by the formula for $E$ in \cite[(2.15)]{ChoiHuangLee2025}, we can write $E(u) = \tilde E(x,u,\nabla u,\nabla^2u)$ for a function $\tilde E\colon\Sigma\times\R\times\R^m\times\R^{m\times m}\to\R$, which satisfies the assumptions of \cref{lem:sigma-estimates} (see \cite[Lemma 4.1]{ChoiHuangLee2025}). Therefore, the estimates follow from the lemma. 
We focus thus here on $\e$, which we explicitly compute.

Let $\eta_u$ be the outward unit co-normal to $\partial\Sigma_u$. Then, 
\[
\eta = \frac{\eta_u - \sk{\eta_u}{\bar\nu} \bar \nu}{\abs{\eta_u - \sk{\eta_u}{\bar\nu} \bar \nu}} = \frac{\eta_u - \sk{\eta_u}{\bar\nu} \bar \nu}{\sqrt{1 - \sk{\eta_u}{\bar\nu}^2}}\,.
\]
Now, let $p$ be a point in $\partial \Sigma$ and choose normal coordinates for $\Sigma$ around $p$, such that $e_m(p) = \eta(p)$ and $\{e_\alpha(p)\}_{\alpha=1,\ldots,m-1}$ is an orthonormal basis of $T_p\partial\Sigma$ at $p$.
Observe that $F_1(x),\ldots,F_{m-1}(x)$ are a (not necessarily orthonormal) basis of $T_{F(x)}\partial\Sigma_u$. Therefore, at such point $p$, we have
\[
\eta_u = \frac{F_m - g^{\alpha\beta}g_{m\beta}F_\alpha}{\abs{F_m - g^{\alpha\beta}g_{m\beta}F_\alpha}} = \frac{F_m - g^{\alpha\beta}g_{m\beta}F_\alpha}{\sqrt{g_{mm} - g^{\a\beta}g_{m\a} g_{m\beta}}}\,,
\]
where $(g^{\alpha\beta})_{\alpha,\beta=1,\ldots,m-1}$ is the metric inverse to $(g_{\alpha\beta})_{\alpha,\beta=1,\ldots,m-1}$. Note that we use indices $\alpha,\beta$ varying on $1,\ldots,m-1$ and indices $i,j$ varying on $1,\ldots,m$.

As a consequence, using that $\sk{\eta_u}{\nu_u}=0$, we have
\begin{align*}
\sk{\nu_u}{\eta} &= \frac{-\sk{\eta_u}{\bar\nu}\sk{\nu_u}{\bar\nu}}{\sqrt{1 - \sk{\eta_u}{\bar\nu}^2}} \\
&=- (\fnu_m - g^{\alpha\beta} g_{m\beta}\fnu_\alpha) \left( \frac{1-g^{ij}\fnu_i\fnu_j}{ g_{mm} - g^{\alpha\beta}g_{m\alpha}g_{m\beta} - (\fnu_m-g^{\alpha\beta}g_{m\beta}\fnu_\alpha)^2} \right)^{\frac 12}\,.
\end{align*}
This implies that (at the point $p$)
\begin{equation}\label{e-form}
\e(u) = \frac{\partial u}{\partial \eta} - \II^{\partial M}(\nu,\nu)u - (\fnu_m - g^{\alpha\beta} g_{m\beta}\fnu_\alpha) \left( \frac{1-g^{ij}\fnu_i\fnu_j}{ g_{mm} - g^{\alpha\beta}g_{m\alpha}g_{m\beta} - (\fnu_m-g^{\alpha\beta}g_{m\beta}\fnu_\alpha)^2} \right)^{\frac 12}\,.
\end{equation}

Now, by \eqref{eq:Ri}, and using that $e_m=\eta$ at $p$, we have that
\begin{align*}
\fnu_m &= u_m + u \sk{ e_m}{\nabla_\nu\bar\nu} + \int_0^u\frac{d^2}{ds^2}\sk{\nabla_{e_m}\psi_\Sigma}{\bar \nu} (u-s) = \frac{\partial u}{\partial \eta} - \II^{\partial M}(\nu,\nu) + \int_0^u\frac{d^2}{ds^2}\sk{\nabla_{e_m}\psi_\Sigma}{\bar \nu} (u-s)\\
&= \frac{\partial u}{\partial \eta} - \II^{\partial M}(\nu,\nu) + O(\norm{u}_{C^0(\Sigma)}^2)\,.
\end{align*}
In particular, using also that $\fnu_i = O(\norm{u}_{C^1(\Sigma)})$ and $g_{ij} = \delta_{ij} + O(\norm{u}_{C^1(\Sigma)})$ from \eqref{eq:Ri} and \eqref{eq:MetricGraph}, we have that $\norm{\e(u)}_{C^0(\partial\Sigma)}\le C\norm{u}^2_{C^1(\Sigma)}$.
Moreover, using \eqref{e-form}, we see that the differential of $\e$ can be written as a function depending on $u,\nabla u, \nabla^2u$, namely $\nabla \e = \tilde\sigma(x,u,\nabla u,\nabla^2u)$, where $\tilde\sigma$ satisfies the assumption of \cref{lem:sigma-estimates}, as a function defined on $\partial\Sigma$ instead of $\Sigma$ (however, note that $\nabla u, \nabla^2u$ are computed on $\Sigma$ and just restricted to $\partial\Sigma$). This implies all the remaining desired estimates on the norms of $\e$. 

Finally, let us prove the estimates on the $L^2$-norms of $E(u)$ and $\e(u)$. Observe that the proof of \cref{lem:sigma-estimates} shows that the estimates $\norm{E(u)}_{C^0(\Sigma)}\le C\norm{u}_{C^1(\Sigma)}\norm{u}_{C^2(\Sigma)}$ and $\norm{\e(u)}_{C^0(\partial\Sigma)} \le C\norm{u}^2_{C^1(\Sigma)}$ hold pointwise, in the sense that $\abs{E(u)}\le C (\abs{u}+\abs{\nabla u})(\abs{u}+\abs{\nabla u} + \abs{\nabla^2u})$ at every $x\in\Sigma$ and $\abs{\e(u)}\le C (\abs{u}+\abs{\nabla u})^2$ at every $x\in\partial\Sigma$. Therefore, we obtain the desired inequalities:
\[
\norm{E(u)}_{L^2(\Sigma)} \le C\norm{(\abs{u}+\abs{\nabla u})(\abs{u}+\abs{\nabla u} + \abs{\nabla^2u})}_{L^2(\Sigma)} \le C \norm{u}_{C^2(\Sigma)} \norm{u}_{W^{1,2}(\Sigma)}\,,
\]
and, using the trace inequality,
\begin{align*}
\norm{\e(u)}_{L^2(\partial\Sigma)} &\le C\norm{(\abs{u}+\abs{\nabla u})^2}_{L^2(\partial\Sigma)} \le C \norm{(\abs{u}+\abs{\nabla u})^2}_{W^{1,2}(\Sigma)}\\
&\le C \left(\norm{(\abs{u}+\abs{\nabla u})^2}_{L^{2}(\Sigma)} + \norm{\nabla(\abs{u}+\abs{\nabla u})^2}_{L^{2}(\Sigma)}\right) \le C \norm{u}_{C^2(\Sigma)}\norm{u}_{W^{1,2}(\Sigma)}\,. \qedhere
\end{align*}
\end{proof}

\subsection{Reversed Poincaré inequality for area-decreasing graphs}
We need the following reversed Poincaré inequality for functions $u$ whose graph has area less than the area of the base hypersurface.

\begin{proposition}[cf. {\cite[Lemma~4.7]{AngenentDaskalopoulosSesum2019}, \cite[Proposition~2.3]{BrendleChoi2019}}] \label{prop:reversedpoincare}
Let $\Sigma\subset M$ be a free boundary minimal hypersurface. There exists $\varepsilon>0$ such that the following result holds.
Let $u\colon\Sigma\to\R$ be a smooth function with $\norm{u}_{C^1(\Sigma)}\le \varepsilon$ and let $\Sigma_u\subset M$ be the graph of $u$ over $\Sigma$ via $\psi_\Sigma$ as in \cref{sec:GraphsViaFoliation}.
If $\area(\Sigma_u)\le\area(\Sigma)$, then
\[
    \int_\Sigma\abs{\nabla u}^2\le C \int_\Sigma\abs{u}^2,
\]
for some constant $C>0$ depending on $M$, $\Sigma$ and $\varepsilon$.
\end{proposition}
\begin{proof}
Recall from \eqref{eq:MetricGraph} that 
\begin{align*}
g_{ij} - g_{ij}^\Sigma &=  u_iu_j- 2\sk{A(e_i,e_j)}{\nu} u + \int_0^u\frac{d^2}{ds^2}\sk{\nabla_{e_i}\psi_\Sigma}{\nabla_{e_j}\psi_\Sigma} (u-s) +{}\\
&\phantom{=} {} + u_i \int_0^u \frac{d}{ds} \sk{\nabla_{e_i}\psi_\Sigma}{\bar\nu} + u_j\int_0^u \frac{d}{ds} \sk{\nabla_{e_j}\psi_\Sigma}{\bar\nu}\\
&= u_iu_j- 2\sk{A(e_i,e_j)}{\nu} u  + u_i a_j(u) + u_ja_i(u) + b_{ij}(u)\,,
\end{align*}
where $a_i,b_{ij}$ are smooth functions depending on $u$, such that $\abs{a_i(u)}\le C\abs{u}$, $b_{ij}(u)\le C{u}^2$.
Therefore, we get (see also \cite[Chapter~2, Section~6]{Simon1983lectures})
\begin{align*}
    \sqrt{\det g_{ij}}
    &= \sqrt{\det g_{ij}^\Sigma}\left[1 + Hu  + \frac 12\abs{\nabla u}^2 + \sk{\nabla u}{\vec a(u)} + b(u)+ (Hu)^2 - \abs{A}^2{u}^2\right] + O(\abs{u}_{C^1}^3)\,,
\end{align*}
where $\vec a (u) = ((g^\Sigma)^{ij}a_i(u))_j$ and $b(u) = \frac12 (g^\Sigma)^{ij}b_{ij}(u)$.
Using that $\Sigma$ is a minimal hypersurface, we thus have
\begin{align*}
    \sqrt{\det g_{ij}}
    &= \sqrt{\det g_{ij}^\Sigma}\left[1 +  \frac 12\abs{\nabla u}^2 + \sk{\nabla u}{\vec a(u)} + b(u) - \abs{A}^2{u}^2 \right]+ O(\abs{u}_{C^1}^3)\,.
\end{align*}
As a consequence, we get that
\[
    \area(\Sigma_u) = \area(\Sigma) + \int_\Sigma\left[\frac 12\abs{\nabla u}^2 + \sk{\nabla u}{\vec a(u)} + b(u) - \abs{A}^2{u}^2 \right] + O(\norm{u}^3_{W^{1,2}(\Sigma)})\,.
\]
Using that $\area(\Sigma_u)\le\area(\Sigma)$, we thus get 
\[
    \norm{\nabla u}_{L^2(\Sigma)}^2 \le C \norm{\nabla u}_{L^2(\Sigma)} \norm{u}_{L^2} + C \norm{u}_{L^2(\Sigma)}^2+ O(\norm{u}^3_{W^{1,2}(\Sigma)})\,,
\]
which implies the desired reversed Poincaré inequality if $\norm{u}_{C^1(\Sigma)}\le \varepsilon$ is sufficiently small.
\end{proof}

\subsection{Graphical mean curvature flow}

Let $u\colon \Sigma\times(T_0,T)\to\R$ be a time-dependent function over~$\Sigma$, and consider the associated family of hypersurfaces $\Sigma_t$ defined as the graph of $u(\cdot,t)$ over $\psi_\Sigma$. In other words, $\Sigma_t$ is the image of the function $F(\cdot,t)\colon\Sigma\to M$ given by  
\[
F(x,t) \eqdef \psi_\Sigma(x, u(x,t))\,.
\]
Then the family of hypersurfaces $\Sigma_t$ moves with velocity  
\[
\frac{d}{dt} F = u_t\, \overline{\nu}\,.
\]
In particular, the normal speed is $u_t \langle \nu_u, \overline{\nu} \rangle$.  
Consequently, the family evolves by graphical mean curvature flow if and only if  
\begin{equation}\label{gmcf}
\begin{cases}
u_t = \displaystyle\frac{H_u}{\langle \nu_u, \overline{\nu} \rangle} &\text{on $\Sigma$}\\
\sk{\nu_u}{\eta} = 0 & \text{on $\partial\Sigma$} \,. 
\end{cases}
\end{equation}

Using the linearization of the mean curvature and the Neumann condition in \eqref{linearization}, we infer that graphical mean curvature flow is equivalent to a solution of the equation 
\begin{equation}\label{gmcf2}
\begin{cases} u_t-\jac_\Sigma u= E(u) & \text{on $\Sigma$}\\
\frac{\partial u}{\partial \eta}-\II^{\partial M}(\nu, \nu)u=\e(u) & \text{on $\partial\Sigma$}\,.
\end{cases}
\end{equation}

First, let us state here, as we need it later, the Schauder estimates for linear parabolic equations as~\eqref{gmcf2}, see for example \cite{Simon1997}, \cite[Theorem 1]{Baderko1998}, and \cite[Theorem~6.30]{GilbargTrudinger2001} for the elliptic case.
\begin{theorem}[Parabolic Schauder estimates]\label{Schauder}
Let $w\colon\Sigma\times(T_0,T)\to\R$ be a solution to the equation 
\[
\begin{cases}
w_t-\jac_\Sigma w=f(x,t) & \text{on $\Sigma$}\\
\frac{\partial w}{\partial \eta}-\II^{\partial M}(\nu, \nu)w= g(x,t) & \text{on $\partial\Sigma$} \,,
\end{cases}
\]
in an interval $t\in (T_0, T)$. Then, for all $t\in (T_0+2, T)$, 
\[
\|w\|_{C^{2,\a}(\Sigma\times(t-1,t))}\le C(\|w\|_{L^2(\Sigma\times(t-2,t))}+\|f\|_{C^{0,\alpha}(\Sigma\times(t-2,t))}+\|g\|_{C^{1,\alpha}(\partial\Sigma\times(t-2,t))})\,,
\]
where $C>0$ is a constant depending only on $\Sigma$ and $M$.
\end{theorem}

Moreover, in order to study the problem \eqref{gmcf2}, we need the following estimates on the parabolic norms $E$ and $\e$.

\begin{proposition}\label{E-estimates} 
For every time-dependent functions $u,v\colon\Sigma\times(T_0,T)\to\R$ with $\norm{u(\cdot,t)}_{C^1(\Sigma)}<r_0$ for all $t\in(T_0,T)$, we have the following estimates for $E$ and $\e$ defined in \eqref{linearization}, for every $t\in(T_0+1,T)$:
\begin{align*}
\|E(u)\|_{C^{0,\a}(\Sigma\times(t-1, t))}&\le C\norm{u}_{C^{1,\alpha}(\Sigma\times (t-1, t))}\|u\|_{C^{2,\alpha}(\Sigma\times (t-1, t))}\\[1ex]
\|E(u)-E(v)\|_{C^{0,\a}(\Sigma\times(t-1, t))}&\le C\|u-v\|_{C^{2,\alpha}(\Sigma\times (t-1, t))}(\norm{u}_{C^{2,\alpha}(\Sigma\times (t-1, t))} + \norm{v}_{C^{2,\alpha}(\Sigma\times (t-1, t))}) \\[1ex]
\|\e(u)\|_{C^{1,\a}(\partial \Sigma\times(t-1, t))}&\le C\norm{u}_{C^{1,\alpha}(\Sigma\times (t-1, t))}\|u\|_{C^{2,\alpha}(\Sigma\times (t-1, t))} \\[1ex]
\|\e(u)-\e(v)\|_{C^{1,\a}(\partial \Sigma\times(t-1, t))}&\le C\|u-v\|_{C^{2,\alpha}(\Sigma\times (t-1, t))}(\norm{u}_{C^{2,\alpha}(\Sigma\times (t-1, t))} + \norm{v}_{C^{2,\alpha}(\Sigma\times (t-1, t))})\,,
\end{align*}
where $C>0$ is a constant depending only on $\Sigma$, $M$ and $r_0$.
\end{proposition}

\begin{proof}
As described in the proof of \cref{prop:TimeSliceEstimates}, by the formulas \cite[(2.15)]{ChoiHuangLee2025} and \eqref{e-form} for $E$ and $\e$, respectively, we have that both $E(u)$ and $\e(u)$ can be written as $\tilde\sigma(x,u,\nabla u,\nabla^2u)$ for $\tilde\sigma\colon \Sigma\times\R\times\R^m\times\R^{m\times m}\to\R$, and restricted to $\partial \Sigma$ in the case of $\e$, which satisfies the assumptions of \cref{lem:sigma-estimates}. Therefore, the proposition follows from the following general result.
\end{proof}

\begin{lemma} \label{lem:sigma-estimates-parabolic}
Let $\tilde\sigma\colon\Sigma\times\R\times\R^m\times\R^{m\times m}\to\R$ be a $C^3$-function. 
Assume that $\tilde\sigma(x,0,0,0) = 0$, $\frac{d}{ds}|_{s=0}\tilde\sigma(x,sa,sb,sc) = 0$ for all $(a,b,c)\in\R\times\R^m\times\R^{m\times m}$, and $\tilde\sigma$ is linear in the variable $c$, namely $\tilde\sigma(x,a,b,c) = \tilde\sigma_1(x,a,b)\cdot c+\tilde\sigma_2(x,a,b)$ for two functions $\tilde\sigma_1,\tilde\sigma_2\colon \Sigma\times\R\times\R^m\to\R$.

For every function $u\colon\Sigma\to\R$, define $\sigma(u)\eqdef\tilde\sigma(x,u,\nabla u,\nabla^2 u)$.  
Then, for every time-dependent functions $u,v\colon \Sigma\times (T_0,T)\to\R$ on $\Sigma$ such that $\norm{u(\cdot,t)}_{C^1(\Sigma)}<r_0$ for all $t\in(T_0,T)$, we have:
\begin{align*}
\norm{\sigma(u)}_{C^{0,\alpha}(\Sigma\times(t-1,t))} &\le C\norm{u}_{C^{1,\alpha}(\Sigma\times(t-1,t))}\norm{u}_{C^{2,\alpha}(\Sigma\times(t-1,t))} \\[1ex]
\norm{\sigma(u)-\sigma(v)}_{C^{0,\alpha}(\Sigma\times(t-1,t))} &\le C(\norm{u}_{C^{2,\alpha}(\Sigma\times(t-1,t))}+\norm{v}_{C^{2,\alpha}(\Sigma\times(t-1,t))}) \norm{u-v}_{C^{2,\alpha}(\Sigma\times(t-1,t))}\,,
\end{align*}
where $C>0$ is a constant that depends on $\tilde\sigma$ and $r_0$.
\end{lemma}
\begin{proof}
Denote $q(u)\eqdef (u,\nabla u,\nabla^2 u)$, in such a way that $\sigma(u) = \tilde\sigma(q(u))$ (where we are omitting the dependence on the variable $x$). 

Let $u,v\colon\Sigma\times (T_0,T) \to\R$ be time-dependent functions on $\Sigma$. Let $(x_1,t_1),(x_2,t_2)\in\Sigma\times(t-1,t)$ and define $f_i(s) = (1-s)q(u(t_i))(x_i) + s q(v(t_i))(x_i)$, for $i=1,2$, which interpolates between $u$ and $v$ at the point $(x_i,t_i)$ for $s\in[0,1]$. Then, we have that
\begin{align*}
(\sigma(u(t_1))(x_1) - \sigma (v(t_1))&(x_1)) - (\sigma(u(t_2))(x_2) - \sigma(v(t_2))(x_2)) = \\
&=\tilde\sigma(f_1(1)) - \tilde\sigma(f_1(0)) - \tilde\sigma(f_2(1)) + \tilde\sigma(f_2(0)) \\
&=\int_0^1 \nabla\tilde\sigma(f_1(s))\left[ f_1'(s) \right]  - \int_0^1\nabla\tilde\sigma(f_2(s))\left[ f_2'(s) \right] \\
&=\int_0^1 \nabla\tilde\sigma(f_1(s))\left[ f_1'(s) - f_2'(s) \right]  - \int_0^1 \left(\nabla\tilde\sigma(f_2(s)) - \nabla\tilde\sigma(f_1(s))\right)\left[ f_2'(s) \right] .
\end{align*}
Now, observe that 
\[
\abs{f_1'(s)-f_2'(s)} = \abs{q((u-v)(t_1))(x_1)-q((u-v)(t_2))(x_2)} \le \norm{u-v}_{C^{2,\alpha}(\Sigma\times(t-1,t))} (\abs{x_1-x_2}^\alpha+ \abs{t_1-t_2}^{\alpha/2}), 
\]
while
\begin{align*}
\abs{f_1(s)-f_2(s)} &\le \abs{q((u)(t_1))(x_1)-q((u)(t_2))(x_2)} + \abs{q((v)(t_1))(x_1)-q((v)(t_2))(x_2)}\\
&\le (\norm{u}_{C^{2,\alpha}(\Sigma\times(t-1,t))}+\norm{v}_{C^{2,\alpha}(\Sigma\times(t-1,t))}) (\abs{x_1-x_2}^\alpha+ \abs{t_1-t_2}^{\alpha/2}).
\end{align*}
Moreover, note that $\abs{\nabla\tilde\sigma(f_1(s))} \le C\abs{f_1(s)}$, where $C$ depends on $\tilde\sigma$ and $r_0$, since $\nabla\tilde\sigma(0) = 0$ and $\tilde\sigma$ is linear in $c$. 
As a result, we get
\begin{align*}
\abs{(\sigma(u(t_1))&(x_1) - \sigma(v(t_1))(x_2)) - (\sigma(u(t_1))(x_1) - \sigma(v(t_1))(x_2))} \le\\
&\le C \left( \sup_{s\in[0,1]}\abs{f_1(s)} \sup_{s\in[0,1]}\abs{f_1'(s)-f_2'(s)} +  \sup_{s\in[0,1]}\abs{f_1(s)-f_2(s)}\sup_{s\in[0,1]}\abs{f_2'(s)} \right) \\
&\le C( \norm{u}_{C^{2,\alpha}(\Sigma\times(t-1,t))}+\norm{v}_{C^{2,\alpha}(\Sigma\times(t-1,t))}) \norm{u-v}_{C^{2,\alpha}(\Sigma\times(t-1,t))}(\abs{x_1-x_2}^\alpha+ \abs{t_1-t_2}^{\alpha/2}).
\end{align*}
This implies that 
\[
\norm{\sigma(u)-\sigma(v)}_{C^{0,\alpha}(\Sigma\times(t-1,t))} \le C(\norm{u}_{C^{2,\alpha}(\Sigma\times(t-1,t))}+\norm{v}_{C^{2,\alpha}(\Sigma\times(t-1,t))}) \norm{u-v}_{C^{2,\alpha}(\Sigma\times(t-1,t))}.
\]

Let us finally estimate $\norm{\sigma(u)}_{C^{0,\alpha}(\Sigma\times(t-1,t))}$. Define $f_i(s) = s q(u(t_i))(x_i)$, for $i=1,2$, such that $f_1' = q(u(t_1))(x_1)$, $f_2' = q(u(t_2))(x_2)$ and both $\abs{f_1'-f_2'}$ and $\abs{f_1-f_2}$ are bounded by $\norm{u}_{C^{2,\alpha}(\Sigma\times(t-1,t)}(\abs{x_1-x_2}^\alpha + \abs{t_1-t_2}^{\alpha/2})$.
Then we have
\begin{align*}
\sigma(u(t_1))&(x_1)- \sigma(u(t_2))(x_2) = \tilde\sigma(f_1(1)) -\tilde\sigma(f_2(1)) \\ 
&= \int_0^1(1-s) \nabla^2\tilde\sigma(f_1(s)) [f_1', f_1'] - \int_0^1(1-s) \nabla^2\tilde\sigma(f_2(s)) [f_2', f_2']\\ 
&= \int_0^1(1-s) \nabla^2\tilde\sigma(f_1(s)) [f_1'-f_2', f_1'] + \int_0^1(1-s) \nabla^2\tilde\sigma(f_1(s)) [f_2', f_1'-f_2'] +\\
&\phantom{=} {}+\int_0^1(1-s) \left\{\nabla^2\tilde\sigma(f_1(s)) - \nabla^2\tilde\sigma(f_2(s))\right\} [ f_2',f_2'] \,.
\end{align*}
Using again that $\nabla^2u$ is a linear function with respect to $\nabla^2u$, similarly as before, we can estimate the right-hand side with
\[
    C\norm{u}_{C^{1,\alpha}(\Sigma\times(t-1,t))} \norm{u}_{C^{2,\alpha}(\Sigma\times(t-1,t))} (\abs{x_1-x_2}^\alpha + \abs{t_1-t_2}^{\alpha/2})\,,
\]
where $C$ depends on $\tilde\sigma$ and $r_0$, which gives the desired estimate.
\end{proof}

The parabolic Schauder estimates in \cref{Schauder} together with \cref{E-estimates} give the following estimates for a solution to the graphical mean curvature flow that will be very useful in what follows.
\begin{lemma}\label{absorption}
Let $u\colon\Sigma\times(T_0, T)\to \R$ be a solution to the graphical mean curvature flow  \eqref{gmcf2} with $\norm{u(\cdot,t)}_{C^1(\Sigma)}<r_0$ for all $t\in(T_0,T)$, then for any $t\in (T_0+2, T)$
\begin{equation}\label{gmcf-Schauder}
\|u\|_{C^{2,\a}(\Sigma\times(t-1,t))}\le C(\|u\|_{L^2(\Sigma\times(t-2,t))}+\|u\|_{C^{1,\alpha}(\Sigma\times(t-2,t))}\|u\|_{C^{2,\alpha}(\Sigma\times(t-2,t))})\,,
\end{equation}
where $C$ depends only on $\Sigma$, $M$ and $r_0$. 

Moreover, there exists an $\varepsilon>0$, depending only on $\Sigma$ and $M$, such that, if $\|u\|_{C^{1,\a}(\Sigma\times(T-R, T))}<\varepsilon$ for some $T-R\ge T_0$, then 
\begin{equation}\label{gmcf-SchauderII}
\|u\|_{C^{2,\a}(\Sigma\times (T-R/2,T))}\le C \|u\|_{L^2(\Sigma\times(T-R, T))} \,,
\end{equation}
where $C>0$ depends only on $\Sigma$ and $M$. 
\end{lemma}

\begin{proof}
Note first that \eqref{gmcf-Schauder} is a direct consequence of the parabolic Schauder estimates and the estimates on $E$ and $\e$ given in \cref{E-estimates}.
Moreover, the scaled version of the Schauder estimate yield for any $\rho>0$, so that $t-\rho\ge T-R$,
\[
\begin{aligned}
&\|u\|^{(\rho/2)}_{C^{2,\a}(\Sigma \times (t-{\rho/2},t))} 
 \le 
C (\|u\|_{L^2(\Sigma\times(t-\rho,t))}+\|u\|^{(\rho)}_{C^{1,\alpha}(\Sigma\times(t-\rho,t))}\|u\|^{(\rho)}_{C^{2,\alpha}(\Sigma\times(t-\rho,t))})\,,
\end{aligned}
\]
where $C$ does not depend on $\rho$ and where by $\|\cdot\|^{(\rho)}$ we mean the scale invariant norms, that is
\[
\begin{split}
\|u\|^{(\rho)}_{C^{2,\a}(\Sigma \times (t-\rho,t))} &\eqdef\|u\|_{C^0(\Sigma\times (t-\rho, t))}
+ \rho\,\|\nabla u\|_{C^0(\Sigma \times (t-\rho,t))} 
+ \rho^2\,\||\nabla^2 u|+|u_t|\|_{C^0(\Sigma \times (t-\rho,t))} +{} 
\\
&\phantom{\eqdef} {} + \rho^{2+\alpha}\,[\nabla^2 u]_{C^\alpha(\Sigma \times (t-\rho,t))} 
+ \rho^{2+\alpha}\,[u_t]_{C^{\alpha/2}(\Sigma \times (t-\rho,t))} 
\end{split}
\]
and the scaled $C^{1,\a}$-norm is similarly defined.

The second part is derived from a general lemma which shows that the $C^{2,\a}$-norm on the right-hand side can be absorbed into the left-hand side, even if it is defined on a larger time interval, provided that it is multiplied by a sufficient small coefficient. This abstract ``absorption lemma'' appears in certain unpublished PDE notes by Leon Simon; here, we present the version that is tailored to our purpose. 
Define the quantity 
\[
S(t, \rho) \eqdef [u]_{C^{2,\alpha}(\Sigma\times(t-\rho,t))}=
[\nabla^2 u]_{C^\alpha(\Sigma \times (t-\rho,t))} 
+ [u_t]_{C^{\alpha/2}(\Sigma \times (t-\rho,t))}\,. 
\]
Thanks to the following interpolation inequality:
\[
\|u\|^{(\rho)}_{C^{2,\alpha}(\Sigma\times(t-\rho,t))}\le C\left(
\rho^{2+\a} \, [u]_{C^{2,\alpha}(\Sigma\times(t-\rho,t))}+
 \, \norm{u}_{L^2(\Sigma\times(t-\rho,t))}\right) \,,
\]
and the Schauder estimate, together with the assumption on the $C^{1,\alpha}$-norm, we have that
\begin{equation}\label{rho-rho/2}
\rho^{2+\alpha} S(t, \rho/2)\le C(\|u\|_{L^2(\Sigma\times(T-R, T))}+\varepsilon\rho^{2+\alpha}S(t,\rho))
\end{equation}
for all $t, \rho$ such that $T-R\le t-\rho<t<T$.
Set
\[
Q \eqdef \sup_{(t-\rho, t)\subset (T-R, T)}
\rho^{2+\alpha} S(t,\rho/2).
\]
Then, inequality \eqref{rho-rho/2}, with $\rho$ replaced by $\rho/2$ yields
\begin{equation}\label{rho-4}
(\rho/2)^{2+\a} S(t, \rho/4)
\le C(\|u\|_{L^2(\Sigma\times(T-R, T))}+ \varepsilon Q )
\end{equation}
for all $t, \rho$ such that $T-R\le t-\rho<t<T$.

Take any interval $(t-\rho/2,t) \subset (T-R, T)$ and write it as the union $(t-\rho/2,t)=(t_2-\rho/4,t_2)\cup(t_1-\rho/4,t_1)\cup(t-\rho/4, t)$, where $t_1=t-\rho/8$, $t_2=t-\rho/4$. Then, by the definition of $S$, one sees that it is subadditive, in the sense that 
\[
S(t, \rho/2)\le S(t, \rho/4)+S(t_1, \rho/4)+S(t_2, \rho/4)\,.
\]
This, together with \eqref{rho-4}, yields
\[
\rho^{2+\a}S(t, \rho/2)\le  C(\|u\|_{L^2(\Sigma\times(T-R, T))}+ \varepsilon Q )
\]
By the arbitrariness of $t, \rho$ we obtain 
\[
Q \le  C(\|u\|_{L^2(\Sigma\times(T-R, T))}+ \varepsilon Q )\,,
\]
which for small enough $\varepsilon$ yields the bound on $Q\le  \|u\|_{L^2(\Sigma\times(T-R, T))}$. By applying the interpolation inequality again the proof is completed.
\end{proof}

\begin{remark}\label{absorptionrmk}
By applying the H\"older interpolation inequality
\[
[f]_\alpha \le [f]_\beta^{\alpha/\beta} \, \|f\|_\infty^{1-\alpha/\beta}\,,\quad \text{for }\beta>\a\,,
\]
we obtain that 
\[
\|u\|_{C^{1,\a}(\Sigma\times(T-R, T))}\le \|u\|_{C^{1,\beta}(\Sigma\times(T-R, T))}^\frac{\a}{\beta}\|u\|^{1-\frac{\a}{\beta}}_{C^{1}(\Sigma\times(T-R, T))}+ \|u\|_{C^{1}(\Sigma\times(T-R, T))}\,.
\]
Therefore, conclusion \eqref{gmcf-SchauderII} of \cref{absorption} still holds if instead of a small $C^{1,\a}$ norm we assume that
\[
\sup\|u(\cdot, t)\|_{C^1(\Sigma)}<\varepsilon\,,\,\,\forall t\in (T-R,T)
\]
and $\|u\|_{C^{1,\beta}(\Sigma\times(T-R, T))}$ is bounded for some $\beta>\a$ and suitably small $\varepsilon$.
\end{remark}

\section{Constructing a family of solutions}\label{sec:construction}

In this section, we construct an $I$-parameter family of ancient solutions coming out of the minimal hypersurface $\Sigma$, where $I=\ind(\Sigma)$. 

Let $\phi\eqdef \sum_{k=1}^I a_k\phi_k e^{-\l_k t}$. We show that we can perturb $\phi$ so that we can create a solution to graphical mean curvature flow, and we can control the perturbation $u$, so that it is of  order $o(e^{-\l t})$, where $\l=\max\{\l_k\st a_k\ne 0\}$. Note that $\phi$ satisfies $\phi_t= J_\Sigma\phi$, and the boundary condition $\frac{\partial \phi}{\partial \eta}-\II(\eta, \eta)\phi=0$. Therefore, in order for $\phi+u$ to solve the graphical mean curvature flow equation~\eqref{gmcf2}, the function $u$ has to satisfy
\begin{equation}\label{gmcf3}
\begin{cases}
u_t-\jac_\Sigma u=E(u+\phi) & \text{on $\Sigma$}\\
\frac{\partial u}{\partial \eta}-\II^{\partial M}(\nu, \nu)u=\e(u+\phi) &\text{on $\partial\Sigma$}\,,
\end{cases}
\end{equation}
where $E$ and $\e$ are as in \eqref{linearization}.
We prove that such a solution exists by applying the contraction mapping theorem. In order to obtain an $I$-parameter family we have to ensure that the constructed solutions are distinct for distinct $\phi$'s. In order to do that, we solve \eqref {gmcf3} under the initial condition that  at time $t=0$, $u$ is orthogonal to all the $\phi_k$'s, for every $k\in\{1,\ldots,I\}$.
In order to show that the contraction mapping can be applied to yield a solution of the problem \eqref{gmcf3}, we first need to obtain good estimates for solutions to the corresponding linear problem, as follows.

\begin{lemma}\label{mainPDE}
Let $f\colon \Sigma\times (-\infty, 0]\to \R$ and $g\colon \partial\Sigma\times (-\infty, 0]\to \R$ be smooth functions, for which we assume $\| f\|_{L^{2, \l}}+\|g\|_{L^{2, \l}}<\infty$, where $\l<0$ is not an eigenvalue of \eqref{eq:JacobiEigenvalueProblem}. Then, there exists a unique solution $w\colon\Sigma\times(-\infty,0]\to\R$ of 
\[
\begin{cases}
w_t-\jac_\Sigma w=f(x,t) & \text{on $\Sigma$}\\
b(w)=g(x,t)& \text{on $\partial\Sigma$} \,,
\end{cases}
\]
where 
$b(w)\eqdef \frac{\partial w}{\partial \eta}-\II^{\partial M}(\nu, \nu)w$, which satisfies
\[
\int_\Sigma w(x,0)\phi_j(x)=0\,,\quad \text{for all $j$ such that }\l_j< \l\,,
\]
and the estimate
\[
\|w\|_{C^{2,\a,  \l}(\Sigma\times(-\infty,0])}\le C\left(\|f\|_{C^{0,\a,\l}(\Sigma\times(-\infty,0])}+\|g\|_{C^{1,\a,\l}(\partial\Sigma\times(-\infty,0])}\right)\,,
\]
where $C>0$ is a constant that depends only on $\Sigma$ and $M$.
\end{lemma}
\begin{proof} We show that the solution can be obtained as the sum of the solutions of two different problems. The first is a parabolic homogeneous problem  with the prescribed boundary condition, which we solve in \cref{homog}. The second is a parabolic problem with homogeneous boundary condition, which we solve in \cref{bhomog}. 

\begin{claim}\label{homog} For $\bar\lambda < \lambda_1$ small enough, the problem
\begin{equation} \label{eq:ClaimBoundaryProblem}
\begin{cases}
u_t-\jac_\Sigma u - \bar\lambda u =0\,& \text{on $\Sigma$}\\
b(u) = g(x,t) &\text{on $\partial\Sigma$}\,,
\end{cases}
\end{equation}
has a unique solution $u\in C^{2,\a}(\Sigma\times(-\infty, 0])$. Moreover, the solution satisfies the estimates
\begin{align}\label{claimL2}
\|u(\cdot, t)\|_{C^0(\Sigma)}&\le C\|g(\cdot, t)\|_{C^0(\Sigma)} \\
\label{claimeq}
\|u\|_{C^{2,\a}(\Sigma\times(t-1,t))}&\le C\|g\|_{C^{1,\a}(\Sigma\times(t-2, t))}\,,
\end{align}
for all $t<0$, where the constant $C>0$ depends only on $\bar\lambda$, $ M$ and $\Sigma$.
\end{claim}
\begin{proof}[Proof of claim]
First we show that, assuming we have a solution, then it satisfies the stated estimates. Notice, that by the Schauder estimates, \cref{Schauder}, a solution satisfies
\[
\norm{u}_{C^{2,\a}(\Sigma\times(t-1,t))} \le C(\norm{u}_{C^0(\Sigma\times(t-2,t))} + \norm{g}_{C^{1,\a}(\partial \Sigma\times(t-2,t))}).
\]
Therefore, the first claim \eqref{claimL2} implies the second one \eqref{claimeq}. We prove a $C^0$-estimate for $u$ by a comparison argument using the first eigenfunction $\phi_1$ of $\jac_\Sigma$, which we know is strictly positive. Using the notation $q=\abs{A}^2+\Ric(\nu, \nu)$, and for any $\d\in \R$, we have 
\begin{equation}\label{int-max-min}
\begin{split}
(\partial_t-\Delta_\Sigma) \left(\frac{u}{\phi_1+\d}\right)&= \frac{u_t-J_\Sigma u}{\phi_1+\d}+\frac{uJ_\Sigma \phi_1}{(\phi_1+\d)^2}+\frac{qu}{\phi_1+\d}-\frac{uq\phi_1}{(\phi_1+\d)^2} + 2\nabla\left(\frac{u}{\phi_1+\d}\right)\cdot  \frac{\nabla \phi_1}{\phi_1+\d}\\
&= \frac{\overline \l u}{\phi_1+\d}+\frac{-\l_1u\phi_1}{(\phi_1+\d)^2}+\frac{qu\d}{(\phi_1+\d)^2}+2\nabla\left(\frac{u}{\phi_1+\d}\right)\cdot \frac{\nabla \phi_1}{\phi_1+\d}\\
&=\frac{(\overline \l-\l_1) u\phi_1+\d u(\overline \l+q)}{(\phi_1+\d)^2}+2\nabla\left(\frac{u}{\phi_1+\d}\right)\cdot \frac{\nabla \phi_1}{\phi_1+\d}\,.
\end{split}
\end{equation}
And at the boundary we have
\begin{equation}\label{bdry-max-min}
\frac{\partial }{\partial \eta}\left(\frac{u}{\phi_1+\d}\right)=\frac{\II^{\partial M}(\nu, \nu) u+g}{\phi_1+\d}-\frac{u\II^{\partial M}(\nu, \nu) \phi_1}{(\phi_1+\d)^2}=\frac{\II^{\partial M}(\nu, \nu) u\d+g(\phi_1+\d)}{(\phi_1+\d)^2}\,.
\end{equation}
Note that $\phi_1>0$ so the constant $h_\Sigma=\max_\Sigma\phi_1/\min_\Sigma\phi_1$ is positive and finite.
Consider first $\d=-\overline \d\min\phi_1$, $\overline \d\in (0,1)$. Then, if  $u/(\phi_1+\d)$  has a maximum that occurs at a boundary point, then by \eqref{bdry-max-min}, at that point\footnote{Observe that here we are using that the boundary of $M$ has positive curvature along $\nu$, i.e., $\II^{\partial M}(\nu,\nu)>0$.}
\begin{equation}\label{u-max}
u\le \frac{g(\phi_1+\d)}{-\d \II^{\partial M}(\nu, \nu)}\le \frac{|g| \max_\Sigma \phi_1}{\bar \d \min_\Sigma \phi_1\II^{\partial M}(\nu, \nu)}\le  |g|\frac{h_\Sigma}{\overline \d\II^{\partial M}(\nu, \nu)}\,.
\end{equation}
Additionally, by \eqref{int-max-min},  it  cannot have a positive interior maximum if
\[
(\bar\l-\l_1)\phi_1+\d(\bar \l+q)<0\,.
\]
Since 
\[
\begin{split}
(\bar\l-\l_1)\phi_1+\d(\bar \l+q)&\le (\bar\l-\l_1)\min_\Sigma \phi_1-\bar \d\min_\Sigma\phi_1(\bar \l+q)
=\min_\Sigma\phi_1(\bar \l-\l_1-\bar\d(q+\bar\l))\,,
\end{split}
\]
this is ensured, if we choose $\bar \d=\frac{\l_1-\bar \l}{-\overline
\l+\max_\Sigma\abs{q}}$.
Estimate \eqref{u-max} now implies 
\[
u\le\frac{C{(-\bar\l+1)}}{ \l_1-\bar \l}\abs{g}\,,\,\,
\]
where $C$ is a constant that just depends on $M$ and $\Sigma$. 

We can similarly show a lower bound by choosing $\d=\overline \d\min\phi_1$, with $\overline \d\in (0,1)$. Then, if  $u/(\phi_1+\d)$  has a minimum that occurs at a boundary point, then, by \eqref{bdry-max-min} at that point
\begin{equation}\label{u-min}
u\ge \frac{g(\phi_1+\d)}{\d \II^{\partial M}(\nu, \nu)}\ge -\frac{|g|(h_\Sigma+\overline \d)}{\bar \d \II^{\partial M}(\nu, \nu)}\,.
\end{equation}
Additionally, by \eqref{int-max-min}, it  cannot have a negative interior minimum if
\[
(\bar\l-\l_1)\phi_1+\d(\bar \l+q)<0\,.
\]
This is ensured, if we choose $\overline \d = \frac{\l_1-\overline\l}{ \max_\Sigma\abs{q}}$.
Estimate \eqref{u-min} now implies 
\[
u\ge -C\frac{\abs{g}}{\l_1-\bar\l}\,,\,\,
\]
where $C$ is a constant that just depends on $M$ and $\Sigma$ and we have thus proved \eqref{claimL2}.

Finally, we show that a solution does exist in the weak sense by using a theorem on variational existence for linear parabolic equations.
Using the operator $Q$ as in \eqref{bilinear}, we set
\[
Q_{\bar\l}(u,v)=Q(u,v)-\int_\Sigma \bar \l uv
\]
and note that a solution of \eqref{eq:ClaimBoundaryProblem} should satisfy
\[
Q_{\bar\l}(u,v)+\int_\Sigma u_t v=\int_{\partial \Sigma} g v\,.
\]
Then, provided $Q_{\bar \l}$ is bounded and coercive, it is known \cite[Chapter~3, Theorem~4.1]{LionsMagenes1972}, that, given any $t_0<0$, there is a unique solution (in the distributional sense) $u\colon\Sigma\times [t_0,0]\to \R$ with $u(\cdot, t_0)=0$.
The a priori estimates we have already derived, which are uniform in $t_0$, allow us to extract a limit, for a sequence $t_0=t_0^i\to -\infty$, which provides the required solution.

It thus remain to show that $Q_{\bar \l}$ is indeed bounded and coercive, that is
\[
Q_{\bar\l}(u, v)\le C \norm{u}_{W^{1,2}(\Sigma)}\norm{v}_{W^{1,2}(\Sigma)}\quad \text{and}  \quad Q_{\bar\l}(u, u)\ge C^{-1}\norm{u}_{W^{1,2}(\Sigma)}^2\,,\quad\text{for some $C>0$}\,.
\]
Boundedness is immediate by using Holder's inequality and the trace theorem:
\[
\norm{u}_{L^2(\partial \Sigma)}\le C\norm{u}_{W^{1,2}(\Sigma)}\,.
\]
In order to show coercivity, we use a refined estimate of the trace that states that for any $\varepsilon>0$, there exists $C_\varepsilon$ such that\footnote{This refined trace inequality follows easily from the proof of the trace theorem in \cite[Section~5.5, Theorem~1]{Evans2010PDE} by using Young's inequality with weight $\varepsilon$ (cf.\ \cite[\S B.2~(d)]{Evans2010PDE}) instead of the standard Young's inequality (cf.\ \cite[\S B.2~(c)]{Evans2010PDE}) in equation~(1).}
\begin{equation}\label{etrace}
\norm{u}_{L^2(\partial \Sigma)}^2\le \varepsilon\norm{\nabla u}_{L^2(\Sigma)}^2+C_\varepsilon\norm{u}_{L^2(\Sigma)}^2\,.
\end{equation}
Note that by \eqref{Ral} we have 
\[
Q_{\bar \l}(u, u)\ge (\l_1-\bar \l)\norm{u}_{L^2(\Sigma)}\,.
\]
Moreover, by letting $q_0=\max_\Sigma q$ and by using the trace theorem with $\varepsilon=(2\max\{\II^{\partial M}\})^{-1}>0$ we also have 
\[
Q_{\bar\l} (u, u)\ge \norm{\nabla u}_{L^2(\Sigma)}^2-(q_0+\bar \l)\norm{u}_{L^2(\Sigma)}^2-\frac12\norm{\nabla u}_{L^2(\Sigma)}^2-C_0 \norm{u}_{L^2(\Sigma)}^2\,,
\]
where $C_0$ depends only on $\Sigma$ and $M$.
Adding the two estimates we obtain
\begin{equation}\label{Tl-coerc}
2Q_{\bar \l}(u, u)\ge \frac12 \norm{\nabla u}_{L^2(\Sigma)}^2+ (-q_0+\l_1-C_0-2\bar \l)\norm{u}_{L^2(\Sigma)}^2
\end{equation}
and, choosing $\bar \l$ small enough so that $-q_0+\l_1-C_0-2\bar \l>0$, we obtain coercivity.
\end{proof}

Next, we solve a parabolic problem with homogeneous boundary condition. 
\begin{claim}\label{bhomog} 
Given  $f\colon\Sigma\times(-\infty,0]\to\R$ and $u\colon\Sigma\to\R$, such that $\|u\|_{L^{2}(\Sigma)}+\|f\|_{L^{2,\l}}< \infty $, with $\l<0$ not an eigenvalue, there exists a unique solution to the problem 
\[
\begin{cases}
v_t-\jac_\Sigma v= f & \text{on $\Sigma$}\\
b(v)=0&\text{on $\partial\Sigma$} \,,
\end{cases}
\]
with $v(\cdot,t)\in L^2(\Sigma)$ for all $t\in (-\infty,0]$ and
\[
\int_\Sigma v(x, 0)\phi_j(x)=-\int_\Sigma u(x)\phi_j(x) \,,\quad \text{for all $j$ such that }\l_j< \l \,.
\]
Moreover, this solution satisfies
\begin{equation}\label{bh-sln}
\|v\|_{L^{2, \l}(\Sigma\times(-\infty,0])}\le C(\|u\|_{L^{2}(\Sigma)}+ \|f\|_{L^{2,\l}(\Sigma\times(-\infty,0])})
\end{equation}
where the constant $C>0$ depends only on $M$ and $\Sigma$.
\end{claim}
\begin{proof}[Proof of claim] 
Note that such a solution should satisfy
\[
\int_\Sigma \partial_t v \phi + \int_\Sigma -\jac_\Sigma v \phi= \int_\Sigma \partial _t v\phi+\int_\Sigma\nabla v\cdot \nabla \phi-\int_\Sigma q uv-\int_{\partial \Sigma}  \II^{\partial M} (\nu,\nu) v\phi=\int_\Sigma f\phi
\]
and therefore, formally, the solution is given by
\[
v(x,t)=\sum_{j=1}^\infty  v_j(t)\phi_j(x)\,,
\]
with 
\[
v_j(t)=
\begin{cases}
e^{-2\l_j t} a_j-\int_t^0e^{\l_j(s-t)} f_j(s)&\text{for $j$ such that  $\l_j<\l$}\\
\int_{-\infty}^te^{\l_j(s-t)} f_j(s)&\text{otherwise}\,,
\end{cases}
\]
where
\[
 f_j(s)=\begin{cases}\l_ja_je^{-2\l_j s}+\int_\Sigma  f(x,s)\phi_j(x) &\text{for $j$ such that $\l_j<  \l$}\\
  \int_\Sigma  f(x,s)\phi_j(x) &\text{otherwise}\,,
\end{cases}
\]
and 
\[
a_j=\begin{cases}-\int u(x)\phi_j (x) &\text{for $j$ such that $\l_j< \l$}\\
  0 &\text{otherwise}\,.
  \end{cases}\
\]
Note that, formally,
\[
\begin{split}
v_t(x,t)&=\sum_{j=0}^\infty  v'_j(t)\phi_j(x)=\sum_{j=0}^\infty   f_j(t)\phi_j(x)-\sum_{j\colon\l_j<\l}\l_je^{-2\l_j t} a_j\phi_j-\sum_{j=0}^\infty\l_j v_j\phi_j\\
&=  f+ \sum_{j=0}^\infty v_j \jac_\Sigma \phi_j =f+\jac_\Sigma v\,.
\end{split}
\]

We now want to show that the above infinite sums are well defined.
Choose $\delta>0$ so that $\l+2\d<0$ and $(-2\d+ \l, \l+2\d)$ does not contain any eigenvalues. For all $j$ such that $\l_j<\l$, using that $\l_j-\l+\d<0$, we estimate
\[
\begin{split}
\abs{v_j -e^{-2\l_j t}a_j}^2&\le \left(\int_t^0  e^{2(\l_j-\l+\d)(s-t)}\right)\left(\int_t^0 e^{2(\l-\d)(s-t)} \abs{f_j(s)}^2 \right) \\
&\le \frac{1}{2(\l-\d-\l_j)} \int_t^0 e^{2(\l-\d)(s-t)} \abs{f_j(s)}^2\,,
\end{split}
\]
and for $j$ such that $\l_j>\l$, using that $\l_j-\l-\d>0$, we estimate
\[
\begin{split}
\abs{v_j}^2 &\le \left(\int_{-\infty}^t  e^{2(\l_j-\l-\d)(s-t)}\right)\left(\int_{-\infty}^t e^{2(\l+\d)(s-t)} \abs{f_j(s)}^2\right) \\
&\le \frac{1}{2(\l_j-\l-\d)}\int_{-\infty}^t e^{2(\l+\d)(s-t)} \abs{f_j(s)}^2\,.
\end{split}
\]
Therefore, since by Parseval's formula  
\[
\sum_{j=0}^\infty \abs{f_j(s)}^2\le 2\|f(\cdot, s)\|^2_{L^2(\Sigma)}+2\sum_{j\colon\l_j<\l}\l_j^2a_j^2e^{-4\l_j s}\le 2\|f(\cdot, s)\|^2_{L^2(\Sigma)}+2\l^2\|u\|^2_{L^2(\Sigma)} e^{-4\l s}\,, 
\]
we obtain
\[
\begin{split}
\sum_{j=0}^\infty v_j^2&\le 2 \|u\|^2_{L^2(\Sigma)}e^{-4 \l t}+\frac{2}{\d}\int_t^0 e^{2(\l-\d)(s-t)}\left(\| f(\cdot, s)\|^2_{L^2(\Sigma)}+\l_1^2 \|u(\cdot, 0)\|^2_{L^2(\Sigma)} e^{-4\l s}\right)+{}\\
&\phantom{=}+\frac 2\d\int_{-\infty}^t e^{2(\l+\d)(s-t)} \| f(\cdot, s)\|^2_{L^2(\Sigma)}\\
&\le 2 \|u\|^2_{L^2(\Sigma)}e^{-4 \l t}+\frac{2}{\d}\|f\|^2_{L^{2,\l}}e^{-2\l t}+\frac{\l_1^2}{-(\l+\d)\d}\|u\|_{L^2(\Sigma)}^2 e^{-2\l t} < \infty\,.
\end{split}
\]
This proves that $v(x,t)$ is well defined and is a solution of the desired problem.
Finally, using Parseval's formula again, we have
\[
\|v(\cdot ,t)\|_{L^2(\Sigma)}^2\le C \|u\|^2_{L^2(\Sigma)}e^{-4 \l t}+C \| f\|^2_{L^{2, \l}} e^{-2\l t},
\]
where $C$ depends only on $\Sigma$, and this implies \eqref{bh-sln}.
\end{proof}

Given now $f, g$ as in the hypotheses of \cref{mainPDE} with $\|f\|_{L^{2,\l}}+\|g\|_{L^{2,\l}}<\infty$, we let $u$ be the  solution of the problem \eqref{eq:ClaimBoundaryProblem} in \cref{homog} with fixed $\overline \l<2\l_1$ and with the given $g$. Then, by \cref{homog} we have that $u$ satisfies \eqref{claimL2} and \eqref{claimeq} with  $C$ depending only on $\Sigma$ and $M$. We also let $v$ be the unique solution given in \cref{bhomog}, with $u=u(\cdot,0)$ (using the aforementioned $u$) and with $f$ replaced by $f-\bar  \l u$. That is
\[
\begin{cases}
v_t-\jac_\Sigma v= f-\bar\l u & \text{on $\Sigma$}\\
b(v)=0&\text{on $\partial\Sigma$} \,,
\end{cases}
\]
and it satisfies the estimate
\begin{equation}\label{vestim}
\|v\|_{L^{2, \l}}\le C\|u(\cdot,0)\|_{L^2(\Sigma)}+ \|f-\bar \l u\|_{L^{2,\l}}\le C(\|f\|_{L^{2,\l}}+\|g\|_{L^{2,\l}(\partial\Sigma\times(-\infty,0])}) \,,
\end{equation}
where in the last inequality we used  \eqref{claimL2} and where $C$ depends only on $M$ and $\Sigma$. We consider now $w=u+v$, with $u$ and $v$ as above. Then
\[
\begin{cases}
w_t-J_\Sigma w=\bar\l u+ f-\bar \l u=f  & \text{on $\Sigma$}\\
b(w)=g&\text{on $\partial\Sigma$} 
\end{cases}
\]
and
\[
\int_\Sigma w(x, 0)\phi_j(x) =0\,,\quad \text{for all $j$ such that }\l_j< \l \,.
\]
Moreover, by \eqref{claimL2} and \eqref{vestim} we have
\[
\|w\|_{L^{2, \l}}\le \|u\|_{L^{2, \l}}+\|v\|_{L^{2, \l}}\le C(\|f\|_{L^{2,\l}}+\|g\|_{L^{2,\l}(\partial\Sigma\times(-\infty,0])})\,.
\]
Finally, by the parabolic Schauder estimates, \cref{Schauder},  $w$ satisfies
\[
\|w\|_{C^{2,\a}(\Sigma\times(t-1, t))}\le C\left(\|w\|_{L^{2}(\Sigma\times(t-2, t))}+\|f\|_{C^{0,\a}(\Sigma\times(t-2, t))}+\|g\|_{C^{1,\a}(\partial \Sigma\times(t-2, t))}\right)
\]
and by multiplying with $e^{\l t}$ we obtain
\begin{align*}
\|w\|_{C^{2,\a,\l}(\Sigma\times(-\infty,0])} &\le C\left( \|w\|_{L^{2,\l}(\Sigma\times(-\infty,0])}+\|f\|_{C^{0,\a,\l}(\Sigma\times(-\infty,0])}+\|g\|_{C^{1,\a,\l}(\partial \Sigma\times(-\infty,0])}\right)\\
&\le C\left(\|f\|_{C^{0,\a,\l}(\Sigma\times(-\infty,0])}+\|g\|_{C^{1,\a,\l}(\partial \Sigma\times(-\infty,0])}\right)\,,
\end{align*}
which is the required estimate of the lemma. This estimate also yields uniqueness by considering the difference of two possible solutions.
\end{proof}

Now we have all the ingredients to prove that we can find a solution for \eqref{gmcf3}.

\begin{theorem} \label{construction} 
Let $\phi=\sum^I_{k=1}a_k\phi_ke^{-\l_k t}$ and let $\l\in (\l_I, 0)$ and $\a\in (0,1)$. There exists $\varepsilon>0$, depending on $\l$ and $\a$ such that, if $\abs{a}\eqdef\left(\sum^I_{k=1}{a_k^2}\right)^\frac12<\varepsilon$, then there is a unique $C^{2,\a}$-function $u\colon \Sigma\times(-\infty,0]\to \R$ solving \eqref{gmcf3} (or, equivalently, $\phi+u$ solves the graphical mean curvature flow equation~\eqref{gmcf2}) and satisfying
\[
\|u\|_{C^{2,\a,\l}(\Sigma\times(-\infty,0])}<\varepsilon
\]
and 
\[
\int_\Sigma u(x, 0)\phi_k(x)=a_k\,,\quad \text{for all $k=1,\ldots,I$}\,.
\]
\end{theorem}
\begin{remark}
Observe that the bound on the weighted H\"older norm implies a strong exponential decay of the solution $u$ to zero.
\end{remark}
\begin{remark}
We remark here that if $J$ is the largest index such that $a_J\ne 0$, then, for  any 
$\l<0$ such that $\l\in (\l_J, 0)$, a solution as in \cref{construction} can be constructed satisfying
\[
\|u\|_{C^{2,\a,\l}(\Sigma\times(-\infty,0])}<\infty
\]
and 
\[
\int_\Sigma u(x, 0)\phi_k(x)=a_k\,,\quad \text{for all $k$ such that }\l_k\le \l_J  \,.
\]
\end{remark}
\begin{proof}
For the given $\l$ and $\a$,  and for $\varepsilon>0$, we consider the space of functions 
\[
\mathcal C_\varepsilon=\{ u\colon\Sigma\times (-\infty, 0]\to \R \st  \|u\|_{C^{2,\a,\l}}<\varepsilon\}
\]
and for $u\in \mathcal C_\varepsilon$, we define $S(u)=v$ to be a solution of 
\[
\begin{cases}
v_t-\jac_\Sigma(v)= E(u+\phi) &\text{on $\Sigma$}\\
b(v)= \e(u+\phi) &\text{on $\partial\Sigma$}\,.
\end{cases}
\]
Note that, by \cref{E-estimates},
\[
\|E(u+\phi)\|_{C^{0,\a,\l}}+\|\e(u+\phi)\|_{C^{1,\a,\l}(\partial \Sigma\times(-\infty, 0])}\le C\|u+\phi\|^2_{C^{2,\a,\l}}\le C(\|u\|^2_{C^{2,\a,\l}} + |a|^2)\,,
\]
therefore, by \cref{mainPDE}, there exists a unique solution  $S(u)=v$
that satisfies
\[
\|S(u)\|_{C^{2,\a,\l}}\le C(\|u\|^2_{C^{2,\a,\l}} + |a|^2)
\]
and 
\[
\int_\Sigma S(u)(x,0)\phi_j(x)=0\,,\quad \text{for all $j$ such that }\l_j< \l\,.\,
\]
Therefore, by choosing $\varepsilon$ small enough, we obtain that for $\abs{a}<\varepsilon$, $S(u)\in \mathcal C_\varepsilon$, and therefore $S$ maps $\mathcal C_\varepsilon$ to itself.

\cref{E-estimates} can be used to also show that $S$ is a contraction map. To see this, note that if $u_1, u_2\in \mathcal C_\varepsilon$, then by \cref{mainPDE}
\begin{align*}
\|S(u_1)-S(u_2)\|_{C^{2,\a,\l}}&\le C\big(\|E(u_1+\phi)-E(u_2+\phi)\|_{C^{0,\a,\l}}+{}\\
&\phantom{\le C (} + \|\e(u_1+\phi)-\e(u_2+\phi)\|_{C^{1,\a,\l}(\partial \Sigma\times(-\infty, 0])}\big)
\end{align*}
and using \cref{E-estimates},  the right hand side is bounded by 
\[
C\|u_1-u_2\|_{C^{2,\a,\l}}(\|u_1+\phi\|_{C^{2,\a,\l}}+\|u_2+\phi\|_{C^{2,\a,\l}})\le C\|u_1-u_2\|_{C^{2,\a,\l}}(\|u_1\|_{C^{2,\a,\l}}+\|u_2\|_{C^{2,\a,\l}}+|a|)\,.
\]
Therefore  choosing $\varepsilon$ small enough, we have
\[
\|S(u_1)-S(u_2)\|_{C^{2,\a,\l}}\le\frac12\|u_1-u_2\|_{C^{2,\a,\l}}\,.
\]
Hence $S$ is a contraction mapping in $C_\varepsilon$. This concludes the proof, with the contraction mapping theorem providing the uniqueness as well.
\end{proof}

\section{Rigidity of the constructed solutions} \label{sec:rigidity}

Considering an $L^2(\Sigma)$-orthonormal basis $\{\varphi_k\}_{k
\ge 1}$ of eigenfunctions of the Jacobi operator \eqref{eq:JacobiEigenvalueProblem} as in \cref{sec:JacobiOp}, we can write any function $u\in L^2(\Sigma)$ as
\[
u=\sum_{k=1}^\infty\langle \phi_k, u\rangle\phi_k \,,
\]
where note that the infinite sum is understood to be converging in $L^2(\Sigma)$. Therefore, even if all the $\phi_k$'s satisfy the boundary condition $b(\phi_k)=\frac{\partial \phi_k}{\partial \eta}-\II^{\partial M}(\nu, \nu)\phi_k=0$, this does not imply that $u$ does as well.
We use the notation
\[
\Phi^+, \Phi^-, \Phi^0
\]
to denote the closures of the positive, negative and null eigenspaces of the Jacobi operator respectively. Note that $\Phi^-$ and $\Phi^0$ are finite dimensional. By abusing notation,  we use the same symbol to denote the corresponding linear projections.
\begin{definition}\label {projections}
Let $u\colon\Sigma\times(-\infty, 0]\to \R$ be a function with $u(\cdot, t)\in L^2(\Sigma)$ for all $t\in (-\infty,0]$. We define $u^+, u^-, u^0$ to be the projections of $u$ on the corresponding eigenspaces, with the projections defined on each $t$-slice, that is 
\[
u^+(\cdot, t)\eqdef\Phi^+(u(\cdot,t))\,,\quad u^-(\cdot,t)\eqdef \Phi^-(u(\cdot,t))\,,\quad u^0(\cdot,t)\eqdef\Phi^0(u(\cdot, t))\,.
\]
Note that  $u^-$ and $u^0$ are the finite sums
\[
u^-=\sum_{k\, :\, \phi_k\in \Phi^-}\langle \phi_k, u\rangle_{L^2(\Sigma)}\phi_k\,,\quad u^0=\sum_{k\, :\, \phi_k\in \Phi^0}\langle \phi_k, u\rangle_{L^2(\Sigma)}\phi_k\,.
\]
As $u^+$ cannot be written as a finite sum, we instead often use the more convenient expression
\[
u^+=u-u^--u^0\,.
\]
\end{definition}

Next we show that if $u$ satisfies a linear equation then its projections as in \cref{projections} satisfy appropriately defined and similar equations.
\begin{lemma}\label{projpde}
Let $u\colon \Sigma\times(-\infty,0]\to\R$ be a solution of
\[
\begin{cases}
u_t-\jac_\Sigma u=f & \text{on $\Sigma$}\\ 
b(u)=g & \text{on $\partial\Sigma$}\,,
\end{cases}
\]
for given $f\colon\Sigma\times(-\infty,0]\to\R$ and $g\colon\partial\Sigma\times(-\infty,0]\to\R$.
Then, any of the three projections of $u$, $U\eqdef\Phi(u)$, where $\Phi=\Phi^+, \Phi^-$ or $\Phi^0$, as defined in \cref{projections}, satisfies
\begin{equation}\label{proj-eqns}
Q(U, \phi)+\int_\Sigma U_t\phi=\int_\Sigma f\Phi(\phi)+\int_{\partial \Sigma}g\Phi(\phi)\,, \quad \forall \varphi\in W^{1,2}(\Sigma)\,,
\end{equation}
where $Q$ is as in \eqref{bilinear}.

\end{lemma}
\begin{proof}
Using the bilinear form $Q$, $u$ satisfies
\begin{equation}\label{Tforu}
Q(u,\phi)+\int_\Sigma u_t\phi=\int_\Sigma f\phi+\int_{\partial \Sigma}g\phi\,.
\end{equation}

Note first that the projections, $u^\pm$ and $u^0$, have the same regularity as $u$. This is because $\Phi^-, \Phi^0$ are finite dimensional and $u^+=u-u^--u^0$, as explained in \cref{projections}. Therefore the
 terms in \eqref{proj-eqns} are well defined. 

Consider first $\Phi=\Phi^-$ or $\Phi^0$. In this case $U=\Phi(u)$ is a finite linear combination of eigenfunctions, and thus at the boundary it satisfies $b(U) = 0$.
Therefore, 
 \[
 Q(U,\phi)=\int_\Sigma -J_\Sigma U \phi= \int_\Sigma -J_\Sigma \varphi U=\begin{cases}\int_{\Sigma}-J_\Sigma\phi u= Q(u, \phi) &\text{if $\phi\in \Phi$}\\
 0&\text{if  $\phi\in \Phi^\perp$} \,.
 \end{cases}
 \]
 Moreover
 \[
 \int_\Sigma \phi U_t =\begin{cases}\int_\Sigma \phi u_t &\text{if $\phi\in \Phi$}\\
 0 &\text{if $\phi\in \Phi^\perp$}\,.
 \end{cases}
 \]
 These, together with \eqref{Tforu}, prove  \eqref{proj-eqns} for $U\in \Phi^-\cup\Phi^0$.

 To prove the last case, we write $u^+=u - u^0-u^-$, and compute 
\[
\begin{split}
Q(u^+,\phi)&= Q(u-u^0-u^-,\phi) = Q(u,\phi) - Q(u^0,\varphi) - Q(u^-,\phi)\,,\\
\end{split}
\]
By the previous case, if $\phi\in\Phi^+$, then $Q(u^0,\phi) = Q(u^-,\phi) = 0$. If $\phi\in(\Phi^+)^\perp = \Phi^0 + \Phi^-$, we get
\[
Q(u^0,\phi)+Q(u^-,\phi) =  Q(u,\Phi^0(\varphi)) + Q(u,\Phi^-(\varphi)) = Q(u,\phi)\,.
\]
In particular, we get that
\[
Q(u^+,\phi) = Q(u,\Phi(\phi)).
\]
Moreover, again using the previous case,
\[
\int_\Sigma u^+_t\phi=\int_\Sigma(u-u^0-u^-)_t\phi=\begin{cases}\int \phi u_t &\text{if $\phi\in \Phi^+$}\\
 0 &\text{if $\phi\in (\Phi^+)^\perp$}\,.
 \end{cases}
\]
By \eqref{Tforu}, this finishes the proof of  \eqref{proj-eqns} for the last case $U = u^+\in \Phi^+$.
\end{proof}

The evolution equations of the projections given in \cref{projpde} allow us to obtain estimates on the evolution of their $L^2$-norms. This is in the spirit of \cite[Lemma 5.5]{AngenentDaskalopoulosSesum2019} adapted here for the bilinear operator $Q$.
\begin{lemma} \label{ADSlemma}
Let $u\colon\Sigma\times(-\infty, 0]\to\R$ be a solution to the free boundary mean curvature flow \eqref{gmcf2}.
There exist constants $\varepsilon>0$ and $C>0$, depending only on $\Sigma$ and $M$, such that if 
\begin{equation}\label{assume1}
\|u(\cdot, t)\|_{{C^{1}(\Sigma)}}<\varepsilon \quad \forall t\le 0 \,,
\end{equation}
then the $L^2$-norms of the projections of $u$, as in \cref{projections}, satisfy:
\begin{align*}
\frac{d \|u^+(\cdot, t)\|_{L^2(\Sigma)}}{dt} &\le -\lambda^+\|u^+(\cdot,t)\|_{L^2(\Sigma)}+ C\|u(\cdot, t)\|_{{C^{2}(\Sigma)}}\|u(\cdot, t)\|_{L^2 (\Sigma)}\,, \\
\frac{d \|u^-(\cdot,t)\|_{L^2(\Sigma)}}{dt} &\ge -\lambda_I\|u^{-}(\cdot, t)\|_{L^2(\Sigma)}-
C\|u(\cdot, t)\|_{{C^{2}(\Sigma)}}\|u(\cdot, t)\|_{L^2(\Sigma)}\,, \\
\left|\frac{d}{dt}\|u^0(\cdot,t)\|_{L^2(\Sigma)}\right| &\le C\|u(\cdot, t)\|_{{C^{2}(\Sigma)}}\|u(\cdot, t)\|_{L^2(\Sigma)}\,,
\end{align*}
where $\lambda^+$ is the smallest positive eigenvalue and $\lambda_I$ is the largest negative eigenvalue.
\end{lemma}
\begin{proof} Let $U\colon\Sigma\times(-\infty, 0]\to\R$ denote any of the three projections, as in \cref{projections}. Then, using \cref{projpde},
\begin{equation}\label{dUdt}
\frac{d}{dt}\frac{1}{2}\|U(\cdot, t)\|^2_{L^2(\Sigma)} =-Q(U, U)+\int_\Sigma E(u) U+\int_{\partial \Sigma} \e(u)U\,.
\end{equation}
The last two integrals on the right-hand side, for any time slice $t$, can be bounded in absolute value, using Cauchy--Schwarz inequality and \cref{prop:TimeSliceEstimates}, by
\begin{equation}\label{IntEe}
\|E(u)\|_{L^2(\Sigma)}\|U\|_{L^2(\Sigma)} +
\norm{\e(u)}_{L^2(\partial\Sigma)} \norm{U}_{L^2(\partial\Sigma)}\le C\|u\|_{C^2(\Sigma)}\|u\|^2_{W^{1,2}(\Sigma)} \,.
\end{equation}

The hypotheses allow us to use \cref{prop:reversedpoincare}, which yields
\begin{equation}\label{W12-L2}
\|u(\cdot,t)\|_{W^{1,2}(\Sigma)}\le C\|u(\cdot, t)\|_{L^2(\Sigma)}\quad \forall t<0 \,.
\end{equation}

Therefore, \eqref{IntEe} together with \eqref{W12-L2}, yield
\[
\left|\int_\Sigma E(u) U+\int_{\partial \Sigma} \e(u)U\right|\le C\|u(\cdot, t)\|_{C^2(\Sigma)}\|u(\cdot, t)\|^2_{L^2(\Sigma)}
\]
Finally, note that  $Q(U, U)\ge \lambda^+\|U\|^2_{L^2}$ for $U\in \Phi^+$, $Q(U, U)\le \lambda_I\|U\|^2_{L^2}$ for $U\in \Phi^-$, and $Q(U, U)=0$ for $U\in \Phi^0$. Plugging these estimates in \eqref{dUdt} finishes  the proof.
\end{proof}

Next we show that any ancient solution emanating from  $\Sigma$ is one of the constructed solutions of \cref{construction}, provided that the convergence to it is sufficiently fast.

\begin{theorem}\label{uniqueness}
Let $u\colon\Sigma\times(-\infty, 0]\to\R$ be a smooth solution to the free boundary mean curvature flow \eqref{gmcf2}.
There exists a constant $\varepsilon>0$, depending only on $\Sigma$ and $M$
, such that if 
\begin{equation} \label{tu:assume1}
\norm{u}_{C^{1,\a}(\Sigma\times(-\infty,0])} < \varepsilon \,,
\end{equation}
and
\begin{equation} \label{tu:assume2}
\liminf_{t\to-\infty}\ \abs{t} \norm{u(\cdot,t)}_{C^0(\Sigma)} = 0 \,,
\end{equation}
then, possibly after a time translation, $u$  coincides with one of the solutions constructed in \cref{construction}.
\end{theorem}

\begin{remark}
Observe that a similar uniqueness result in the case without boundary can be found in \cite[Theorem~4.1]{ChoiMantoulidis2022}. Indeed, the proof of our theorem follows similar lines, given our previous results in the free boundary setting.
Note that the assumptions in \cite[Theorem~4.1]{ChoiMantoulidis2022} are different, but turn out to be stronger than ours. Indeed, assumptions \cite[(4.1), (4.2)]{ChoiMantoulidis2022}:
\[
\norm{u}_{C^{1,\a'}(\Sigma\times(-\infty,0])} < \infty \quad \text{and}\quad \|u(\cdot, t)\|_{{C^{1}(\Sigma)}}<\varepsilon \quad \forall t\le 0 \,,
\]
imply \eqref{tu:assume1}, for $\a<\a'$, thanks to \cref{absorptionrmk}. Moreover, the assumption \cite[(4.3)]{ChoiMantoulidis2022}:
\[
\|u\|_{L^1(\Sigma\times(-\infty, 0])}<\infty
\]
implies \eqref{tu:assume2}, thanks to \cref{absorption} \eqref{gmcf-SchauderII} (indeed, note that the $L^1$- and $L^2$-norms are comparable since $\Sigma$ is compact and $\norm{u}_{C^0(\Sigma)}\le 1$).
\end{remark}

\begin{proof}
Given all our previous results developed in the free boundary setting, the proof of this theorem now follows the steps of \cite[Theorem 4.1]{ChoiMantoulidis2022}. We therefore only include a sketch of the proof here pointing out the adjustment for the boundary case. Throughout the proof, $C>0$ denotes a constant that depends only on $\Sigma$ and $M$.

Due to the hypotheses on $u$, we can invoke the consequence of the parabolic Schauder estimates given in \cref{absorption} to deduce that  
\begin{equation}\label{f1-est}
f(t)\eqdef\|u(\cdot, t)\|_{C^{2,\a}(\Sigma)}\le C\varepsilon\,.
\end{equation}
First we claim that we can apply the ODE lemma \cite [Lemma B.1]{ChoiMantoulidis2022} (a refined version of \cite[Lemma A.1]{MerleZaag1998}) to deduce that 
\begin{equation}\label{dommode}
\|u^+(\cdot,t)\|_{L^2(\Sigma)}+\|u^0(\cdot,t)\|_{L^2(\Sigma)}\le Cf(t)\|u^-(\cdot,t)\|_{L^2(\Sigma)}\,.
\end{equation}
To see this, we note that the quantities $x=\|u^0(\cdot,t)\|_{L^2(\Sigma)}$, $y=\|u^+(\cdot,t)\|_{L^2(\Sigma)}$ and $z=\|u^-(\cdot,t)\|_{L^2(\Sigma)}$, satisfy the hypotheses of \cite [Lemma B.1]{ChoiMantoulidis2022}, due to their evolution equations given in \cref{ADSlemma} together with the smallness of $f(t)$ and the fact that $\liminf_{s\to-\infty} y(s)=0$, which follows from the assumption \eqref{tu:assume2}. \cite [Lemma B.1]{ChoiMantoulidis2022} will then yield \eqref{dommode} provided that we exclude that
\begin{equation}\label{y+z}
y+z\le Cf(t) x\,.
\end{equation}
Indeed, \eqref{y+z} together with \cref{absorption} would imply that 
\begin{equation}\label{C2-L2}
f(t)\le C \norm{u}_{L^2(\Sigma\times(t-1,t))} \le  C \|u^0\|_{L^2(\Sigma\times(t-1, t))}\le  \max_{s\in (t-1, t)}\|u^0(\cdot, s)\|_{L^2(\Sigma)}\,.
\end{equation}
Using the evolution of $x$ provided in \cref{ADSlemma}, this implies that $V(t)\eqdef \max_{s\in (t-1, t)}\|u^0(\cdot, s)\|_{L^2(\Sigma)}$ satisfies $V'\le C V^2$. Integrating, this shows that $V(t)\ge C\abs{t}^{-1}$ for every $t\in (-\infty,0]$ sufficiently small, which contradicts assumption \eqref{tu:assume2}. As a result, we have that \eqref{dommode} holds.

Next, we can apply  \cite[Claims 4.5, 4.6 and (4.19)]{ChoiMantoulidis2022} verbatim, with the proofs unaltered. Observe that the integrability of $f(t)$ used in \cite[Claim~4.6]{ChoiMantoulidis2022} follows from \cite[Claim~4.5]{ChoiMantoulidis2022}.\footnote{In \cite{ChoiMantoulidis2022}, the integrability of $f(t)$ follows more directly from the assumption $\norm{u}_{L^1(\Sigma)}<\infty$.}
As a result, we get that there exists $J\in\{1,\dots, I\}$ such that 
\[
f(t)\le C e^{-\lambda_J t}\quad \forall t<0\,,
\]
and, defining $u^\tau(\cdot, t)\eqdef u(\cdot, t-\tau)$ for $\tau>0$, we have
\[
\limsup_{\tau\to\infty} e^{-\l_J\tau} \|\Phi^-({u^{\tau}})(\cdot, 0)\|_{L^2(\Sigma)}>0 \,.
\]
Let $a_k\eqdef\langle u^\tau(\cdot,0),\phi_k(\cdot) \rangle_{L^2(\Sigma)}$ for $k=1,\dots, I$ and 
set $\phi^\tau \eqdef \sum_{k=1}^Ia_k\phi_ke^{-\l_kt}$. We want to show that, for $\tau$ sufficiently large, $u^\tau-\phi^\tau$ is the solution we constructed in \cref{construction} with $\phi=\phi^\tau$. 
First, observe that
\[
\abs{a} = \left(\sum_{k=1}^I a_k^2\right)^{\frac 12}\le \|u^\tau(\cdot,0)\|_{L^2(\Sigma)}=\|u(\cdot,-\tau)\|_{L^2(\Sigma)}\le Ce^{\l_J \tau}\,.
\]
Moreover, by the $L^2$-estimate given in \cite[Claim 4.7]{ChoiMantoulidis2022}, together with the Schauder estimates \cref{Schauder}, we have
\[
\|u^\tau -\phi^\tau \|_{C^{2,\a,\l}}\le C e^{2\l_J\tau}\,,\quad\text{ for some }\l\in (\l_I, 0)\,.
\]
Therefore, we can choose $\tau$ sufficiently small such that $\abs{a}<\varepsilon$ and $\|u^\tau -\phi^\tau \|_{C^{2,\a,\l}}<\varepsilon$, for $\varepsilon$ as in \cref{construction}, which shows that $u^\tau-\varphi^\tau$ coincides with the unique solutions constructed in \cref{construction}.
\end{proof}

\begin{corollary} \label{COR} If $\Sigma$ is nondegenerate, then \cref{uniqueness} holds without the assumption \eqref{tu:assume2}.
\end{corollary}

\begin{proof}
This follows directly from the proof of \cref{uniqueness}. Indeed, assumption \eqref{tu:assume2} is used there only to exclude \eqref{y+z} in the use of \cite[Lemma B.1]{ChoiMantoulidis2022}. However, if the kernel of the Jacobi operator on $\Sigma$ is empty, then $x=0$ trivially and therefore \eqref{y+z} cannot occur for a nontrivial solution.
\end{proof}

\begin{remark}
The nondegeneracy assumption in the above corollary can be weakened to integrability, following the argument of \cite[Lemma 4.9]{ChoiMantoulidis2022}. In that case, one must allow for convergence to a nearby minimal hypersurface to $\Sigma$ in the conclusion, but the result still holds.
In fact, the assumptions cannot be weakened further as there are examples of mean-convex ancient solutions backward converging to degenerate minimal hypersurfaces slower than exponentially.
\end{remark}

\section{Constructing mean-convex ancient solutions}\label{sec:meanconvex}

In this section, we show an alternative way to construct mean-convex ancient solutions of the free boundary mean curvature flow coming out of the minimal hypersurface $\Sigma$, different from the general method presented in \cref{sec:construction}. For mean-convex solutions there is a more geometric way of constructing them, by first constructing sub- and supersolutions via varying the minimal hypersurface by the first eigenvalue. We remark here that varying the minimal hypersurface by the first eigenvalue is not as straightforward as in the closed case, as preserving the orthogonality condition is far from trivial. We achieve this by a delicate application of the implicit function theorem,  inspired by \cite{Ambrozio2015}*{Proposition~10} (see also \cite{MaximoNunes2013}*{Proposition~5.1}, \cite{Mazurowski2022}*{Section~4.5}), which allows us to prove that it is possible to foliate a neighborhood of $\Sigma$ by free boundary mean-convex hypersurfaces (in the sense that their mean curvature vector points away from $\Sigma$). 

We in fact prove a more general version in the following \cref{prop:SurfKthEigenf}.  We assume that $\Sigma$ is  a properly embedded, smooth, free boundary hypersurface in $M$, not necessarily minimal. In the proposition we consider the quantities $\overline \nu$ and $\psi_\Sigma$ as in \cref{sec:GraphsViaFoliation}. Namely, $\overline \nu$ is a vector field extending, in $M$, a choice of unit normal to $\Sigma$, such that $\bar\nu$ is tangent to $\partial M$ and $\psi_\Sigma$ is the flow of $\overline \nu$ passing through $\Sigma$ at time zero.  

\begin{proposition}\label{prop:SurfKthEigenf}
Let $\Sigma$ be a properly embedded, smooth, free boundary hypersurface in $M$ and fix $k = 1,\ldots,\ind(\Sigma)$. Then, for every $\rho\in(\l_I, 0)$ and $\d>0$ , there exist $T\in\R$ and a  $C^{2,\a}$-function $u\colon \Sigma\times (-\infty,T)\to\R$ with $\|u(\cdot,t)\|_{C^{2,\a}(\Sigma)}<\delta$ for all $t\in(-\infty,T)$, and satisfying the following property. Consider the function
\[
w(x,t) = \varphi_k e^{-\lambda_kt}  + u(x,t)\,,
\]
where recall that $\phi_k$ is an eigenfunction corresponding to the eigenvalue $\l_k$.
Then, the hypersurface
\begin{equation} \label{eq:GraphicalSurf}
\Sigma_t=\Sigma_{w(\cdot,t)}\eqdef \{\psi_\Sigma(x,w(x,t)) \st x\in\Sigma\}
\end{equation}
is free boundary with mean curvature $H_w$ satisfying
\begin{equation} \label{eq:AlmostMC}
\frac{H_w - H_\Sigma}{\sk{\nu_w}{\bar\nu}} =  -\lambda_k \varphi_k e^{-\lambda_kt}  - \rho u \,,
\end{equation}
where $\nu_w$ is the upward pointing normal to $\Sigma_t$. 
Moreover, 
\begin{equation} \label{eq:EstUUt}
\abs{u(x,t)}\le C e^{-2\lambda_k t}\,,\quad \abs{u_t(x,t)}\le C e^{-2\lambda_k t}\,.
\end{equation}
In particular, $\Sigma_t$ converges exponentially fast to $\Sigma$ as $t\to-\infty$. 
\end{proposition}

\begin{remark}
Observe that we are only considering negative eigenvalues, since in this way  $\varphi_{k}e^{-\l_k t}$ converges to $0$ as $t\to-\infty$.

The idea of the proposition is that, if $\Sigma$ is minimal, we expect the existence of an ancient mean curvature flow whose graph over $\Sigma$, at first order, looks like $\varphi_k e^{-\lambda_kt}$ for $t\in(-\infty,T)$, with mean curvature, at first order, being $-\lambda_k\varphi_k e^{-\lambda_kt}$. Our proposition does not directly give this ancient solution, but it adds a correction $u$ to the first order approximation  to satisfy the orthogonality condition at the boundary. Indeed, notice that a priori the graph of $\varphi_k e^{-\lambda_kt}$ over $\Sigma$ is not necessarily orthogonal to the boundary of~$M$.
\end{remark}
\begin{remark}
Assume that $\Sigma$ is an unstable free boundary minimal hypersurface.
Note that $\varphi_k$ changes sign for all $k>1$, while $\varphi_1>0$. Therefore, as observed in \cite[Section~2.2]{HaslhoferKetover2025}, the graph of $\varphi_1 e^{-\l_1 t}$ is a mean-convex foliation of a neighborhood of $\Sigma$. However, this is not necessarily free boundary, as remarked above. In \cref{prop:SurfKthEigenf}, we find a lower order correction such that the hypersurfaces $\Sigma_t$ form a smooth \emph{free boundary} mean-convex foliation of a neighborhood of $\Sigma$.
\end{remark}
\begin{remark}\label{w-norm}
The function $w$ constructed in \cref{prop:SurfKthEigenf} satisfies 
\[
\|w(\cdot, t)\|_{C^{2,\a}(\Sigma)}\le Ce^{-\l_k t}\quad\forall t\in (-\infty, T)\,.
\]
Indeed, since $J_\Sigma(w)+ E(w)=-\l_k\phi_k e^{-\l_k t}-\rho u$, the elliptic Schauder estimates, together with \cref{E-estimates}, yield 
\[
\begin{split}
\|w\|_{C^{2,\a}(\Sigma)}&\le C(\|w\|_{C^0(\Sigma)}+\|E(w)+\l_k\phi_ke^{-\l_k t}+\rho u\|_{C^{0,\a}(\Sigma)}+\|\e(w)\|_{C^{1,\a}(\partial\Sigma)})\\
&\le C(\|w\|_{C^0(\Sigma)}+\|w\|_{C^{1,\a}(\Sigma)}\|w\|_{C^{2,\a}(\Sigma)}+e^{-\l_k t}+\rho\| u\|_{C^{0,\a}(\Sigma)})\\
&\le C(e^{-\l_k t} + (\|w\|_{C^{1,\a}(\Sigma)} + \rho)\|w\|_{C^{2,\a}(\Sigma)}) \,.
\end{split}
\]
Moreover, recall that $\norm{w}_{C^{2,\a}(\Sigma)}\le C e^{-\l_kt} + \norm{u}_{C^{2,\a}(\Sigma)} < C e^{-\l_kt} + \delta$. As a result, for $\d>0$ and $\rho>0$ sufficiently small, and for time $t\in(-\infty,T)$ sufficiently small, we can absorb the term $(\|w\|_{C^{1,\a}(\Sigma)} + \rho)\|w\|_{C^{2,\a}(\Sigma)}$ on the left-hand side and obtain the result.
\end{remark}
\begin{proof}
Observe that $\rho\in(\l_I,0)$ is not an eigenvalue of the Jacobi operator on $\Sigma$, namely the following problem does not admit nontrivial solutions 
\[
\begin{cases}
-\jac_\Sigma v  = \rho v & \text{on $\Sigma$}\\
\frac{\partial v}{\partial \eta} = \II^{\partial M} (\nu,\nu) v  & \text{on $\partial \Sigma$} \,.
\end{cases}
\]

Now, it is convenient to adopt the change of variable $\tau = e^{-\lambda_k t}$ and look for $\tau$ defined in $(-\delta,\delta)$ for some $\delta>0$.
Then, let us define
\[
h_{\varphi_k\tau+u}(x,\tau) \eqdef -\lambda_k\varphi_k \tau\sk{\nu_{\varphi_k\tau+u}}{\bar\nu}\,,
\]
and the map $\Phi\colon (-\delta,\delta)\times (B(0,\delta)\subset C^{2,\alpha}(\Sigma))\to C^{0,\alpha}(\Sigma)\times C^{1,\alpha}(\partial \Sigma)$ as\[
\Phi(\tau,u) = \left(H_{\varphi_k\tau + u}  - H _\Sigma - h_{\varphi_k\tau+u} + \rho u \sk{\nu_{\varphi_k\tau+u}}{\bar\nu}, \sk{\nu_{\varphi_k\tau+u}}{\eta_{\varphi_k\tau+u}} \right) \,.
\]
Recall that $\nu_{\varphi_k\tau+sv}$ coincides with $\bar\nu$ for $\tau=0$ and $s=0$, and $\frac{d}{ds}|_{s=0} \sk{ \nu_{sv}}{\bar\nu} = 0$ (see \cref{sec:GraphsViaFoliation}).
Therefore, by \cite{Ambrozio2015}*{Proposition~17} (see also proof of Proposition~10 therein), we have that
\[
D\Phi_{(0,0)} (0,v) = \frac{d}{ds}\big|_{s=0} \Phi(0,sv) = \left( \jac_\Sigma v +\rho v , -\frac{\partial v}{\partial \eta} + \II^{\partial M}(\nu,\nu) v \right) \,.
\]
Then, $D\Phi_{(0,0)}(0,\cdot)\colon C^{2,\alpha}(\Sigma)\to C^{0,\alpha}(\Sigma)\times C^{1,\alpha}(\partial\Sigma)$ is invertible. Indeed, it is injective by the choice of $\varepsilon$, and it is surjective by \cite{LadyzhenskayaUraltseva1968}*{Theorem~3.2, pp. 137}.

As a result, by the implicit function theorem (possibly taking $\delta>0$ smaller), there exists a function $u\colon (-\delta,\delta)\to B(0,\delta)\subset C^{2,\alpha}(\Sigma)$ such that $u(0)=0$ and $\Phi(\tau,u(\tau)) = 0$ for all $\tau\in(-\delta,\delta)$. Note that, since $\Phi$ is smooth, then $u$ is smooth as well.
Let us define $w(x,\tau)=\varphi_k(x)\tau+u(\tau)(x)$, then the hypersurface $\Sigma_\tau = \Sigma_{w(\cdot,\tau)}$ is a free boundary hypersurface with mean curvature $H_w$ satisfying \eqref{eq:AlmostMC}.

Let us now compute the derivative of  
\[
0 = \Phi(\tau,u(\tau)) = \left(H_w - H_\Sigma - h_{w} + \rho u\sk{\nu_{w}}{\bar\nu}, \sk{\nu_w}{\eta_w}\right)
\]
with respect to $\tau$ at $\tau=0$. Note that, at $\tau = 0$, 
\[
h_w' = -\lambda_k\varphi_k.
\]
Thus, defining $w' = \frac{\partial w}{\partial \tau}\big|_{\tau=0} = \varphi_k + u'$ we get
\begin{align*}
0 &= \left(\jac_\Sigma w' - (h_{w})' +\rho u', -\frac{\partial w'}{\partial \nu} + \II^{\partial M}(\nu,\nu)w'\right)= \left(\jac_\Sigma u' +\rho u', -\frac{\partial u'}{\partial \nu} + \II^{\partial M}(\nu,\nu)u'\right) \,,
\end{align*}
where in the second equality we used that $\varphi_k$ is an eigenfunction of the Jacobi operator with respect to eigenvalue $\lambda_k$.
This implies that
\[
\begin{cases}
-\jac_\Sigma u'  = \rho u' & \text{on $\Sigma$}\\
\frac{\partial u'}{\partial \nu} = \II^{\partial M}(\nu,\nu)u'  & \text{on $\partial\Sigma$}\,,
\end{cases}
\]
namely $u'$ is an eigenfunction of the Jacobi operator relative to eigenvalue $0$ and it is therefore equal to $0$ by the nondegeneracy assumption.
As a result, since $u(0)=u'(0)=0$ and $u$ is smooth ($C^2$ would be sufficient), there exists $C>0$ such that $\abs{u(\tau)}\le C\tau^2$ and $\abs{u'(\tau)}\le C\tau$. By changing variable back to $t$, we get the desired estimates \eqref{eq:EstUUt} on $\abs{u}$ and $\abs{u_t}$, concluding the proof.
\end{proof}

From now on we assume that $\Sigma$ is a free boundary minimal hypersurface with index $I>0$. 
We construct sub- and supersolutions to mean curvature flow that are close to the family $\{w(\cdot,t)\}_t$ constructed in \cref{prop:SurfKthEigenf} for the case $k=1$, associated to the first eigenfunction $\phi_1$. Note that this $w$ satisfies
\begin{equation}\label{wt}
\begin{cases}
\displaystyle   w_t= -\l_1\phi_1 e^{-\l_1t}+  u_t= \frac{H_{w}}{\langle \nu_w, \overline\nu\rangle}+\rho u+ u_t & \text{on $\Sigma$}\\
\sk{\nu_w}{\bar\nu} = 0 & \text{on $\partial\Sigma$} \,,
\end{cases}
\end{equation}
and therefore, comparing with \eqref{gmcf}, the evolving hypersurfaces $\Sigma_{w(\cdot, t)}$ ``almost" satisfy mean curvature flow.
Let us now construct sub- and supersolutions by slightly reparametrizing the time.
\begin{lemma}\label{sub/super} Given $\l\in(\l_I, 0)$, we define the families
\[
w^\pm(x,t)= w(x, t\pm e^{-\l t})\,,
\]
where $\{w(\cdot,t)\}_t$ is the function constructed in the \cref{prop:SurfKthEigenf} for the case $k=1$, associated to the eigenfunction $\phi_1>0$.
Then there exists $t_\l\in\R$ such that $w^+$ is a supersolution and $w^-$ is a subsolution of mean curvature flow for all $t<t_\l$, that is
\[
\frac{d w^+}{dt}\ge \frac{H_{w^+}}{\langle \nu_{w^+}, \overline\nu\rangle}\quad \text{ and }\quad \frac{d w^-}{dt}\le \frac{H_{w^-}}{\langle \nu_{w^-}, \overline\nu\rangle}\,.
\]
\end{lemma}
\begin{remark}\label{sub/sup rmk} Note that \cref{sub/super} is also true for $\phi_1<0$, reversing the roles of $w^+$ and $w^-$.
\end{remark}
\begin{proof}
Let $w_f(x,t)= w(x, t+f(t))$. Then, using \eqref{wt},
\[
\begin{split}
\frac{d w_f}{dt}= \frac{\partial w}{\partial t} (1+f'(t))= \left(\frac{H_{w_f}}{\langle \nu_{w_f}, \overline\nu\rangle}+\rho u+\frac{\partial u}{\partial t}\right)(1+f'(t))\,,
\end{split}
\]
which implies
\begin{equation}\label{w-f}
\frac{d w_f}{dt}-\frac{H_{w_f}}{\langle \nu_{w_f}, \overline\nu\rangle}=-\l_1\phi_1 e^{-\l_1 (t+f(t))} f'(t)+ \left(\rho u+\frac{\partial u}{\partial t}\right)(1+f'(t))\,.
\end{equation}
Recall that, by   \cref{prop:SurfKthEigenf}, $|u|+|u_t|\le Ce^{-2\l_1 t}$ and   $\phi_1>0$. Therefore, for any $\l\in (\l_1, 0)$, if we choose $f(t)= \pm e^{-\l t}$, we obtain that for  
 all $t$ sufficiently small the right-hand side of \eqref{w-f} is positive and negative respectively. 
Hence, the result follows for $w^\pm=w_{\pm e^{-\l t}}$.
\end{proof}

\begin{theorem}\label{MCconstruction}
For the eigenfunction $\phi_1$, corresponding to the first eigenvalue $\l_1$, there exists a graphical ancient solution $\{\Sigma_t\}_{t\in (-\infty, T)}$, given as graphs of a function $v\colon\Sigma\times(-\infty,T)\to\R$ satisfying 
\[
\|v(\cdot, t)\|_{C^{2,\a}(\Sigma)}\le Ce^{-\l_1 t}\,,
\]
for a constant $C>0$ depending only on $M$ and $\Sigma$,
and 
\[
\lim_{t\to -\infty} v(x, t) e^{\l_1 t}=\phi_1(x)\,.
\]
This solution is mean convex.
\end{theorem}

\begin{proof}
Let $w(x,t)=\phi_1 e^{-\l_1 t}+ u(x,t)$ be as in  \cref{prop:SurfKthEigenf},  and $w^\pm$ the corresponding barriers constructed in \cref{sub/super} for some $\l\in (\l_1, 0)$. Without loss of generality assume that  $\phi_1>0$ (the case $\phi_1<0$ being similar, see \cref{sub/sup rmk}). We claim that we can choose $t_\l$,  possibly smaller than in \cref{sub/super}, so that we also have
\begin{equation} \label{eq:IneqWs}
w^+(x, t)\ge w(x,t)\ge w^-(x,t)\quad\forall t<t_\l\,.
\end{equation}
Indeed, using that $\abs{u_t(x,t)}\le Ce^{-2\l_1 t}$ by \cref{prop:SurfKthEigenf},  we get
\[
w^+(x, t)- w(x,t)=\phi_1 e^{-\l_1 t}( e^{-\l_1 e^{-\l t}}-1) + u(x, t+ e^{-\l t})- u(x, t)\ge -\l_1\phi_1e^{-(\l_1+\l) t}- C e^{-(2\l_1+\l) t}\,,
\]
which is clearly positive for $t$ small enough. Similarly we can argue for $w^-$ obtaining \eqref{eq:IneqWs}.

Now, let $t_\lambda>t_0>t_1>t_2>\ldots \to-\infty$ be a decreasing sequence of times and, for every $i\in\N$, let $\Sigma^i = \Sigma_{t_i} = \Sigma_{w(\cdot,t_i)}$ be the hypersurface as in \eqref{eq:GraphicalSurf} (see also \cref{sec:GraphsViaFoliation}).
Moreover, let $\{\Sigma^i_t\}_{t\in [t_i,T_i)}$ be the mean curvature flow starting from $\Sigma^i$ at time $t_i$. For $T_i>t_i$ sufficiently close to $t_i$, the mean curvature flow is smooth and $\Sigma^i_t$ can be written as the graph of a function $v^i(x,t)$ over $\Sigma$, namely $\Sigma^i_t = \Sigma_{v^i(\cdot,t)}$. Here observe that $v^i(\cdot,t_i) = w(\cdot,t_i)$. 

Note that $v^i$ satisfies the graphical mean curvature flow  \eqref{gmcf2} and by  \cref{sub/super} we have that 
\[
w^-(x,t)\le v^i(x, t)\le w^+(x,t)\quad \forall t\in [t_i, T_i)\,.
\]
Therefore, we have $\|v^i(\cdot,t)\|_{C^0(\Sigma)}\le \|w^+(\cdot, t)\|_{C^0(\Sigma)} \le C e^{-\l_1 t}$ for all $t\in [t_i,T_i)$, for a constant $C>0$ that depends only on $M,\Sigma$. Using the Schauder estimates \cref{Schauder}, together with the estimates of \cref{E-estimates}, we have
\[
\|v^i\|_{C^{2,\a}(\Sigma\times(t-1,t))}\le C(\|v^i\|_{L^2(\Sigma\times(t-2,t))}+\|v^i\|_{C^{1,\a}(\Sigma\times(t-2,t))}\|v^i\|_{C^{2,\a}(\Sigma\times(t-2,t))}) \,, 
\]
where $C>0$ here does not depend on $t_i$. Therefore, as long as  $\|v^i\|_{C^{1,\a}(\Sigma\times(t-2,t))}$ is small enough (depending on $M$ and $\Sigma$), by \cref{absorption}, we obtain
\[
\|v^i\|_{C^{2,\a}(\Sigma\times(t-1,t))}\le C\|v^i\|_{L^2(\Sigma\times(t-2,t))}\le Ce^{-\l_1 t}\,.
\]
Since, by Remark \ref{w-norm}, $\|v^i(\cdot, t_i)\|_{C^{2,\a}(\Sigma)}$ is initially of the order $e^{-\l_1 t_i}$, we obtain that the $v_i$'s are indeed defined on intervals $[t_i,T)$, where $T$ is independent of $i$, and they satisfy the uniform estimate
\[
\|v^i\|_{C^{2,\a}(\Sigma\times(t-1,t))}\le Ce^{-\l_1 t}\quad \forall t\in [t_i+1, T)\,.
\]
Therefore, we can extract a subsequence of $\{v_i\}$ that converges in $C^{2,\a'}_{\loc}(  \Sigma\times(-\infty, T))$, for all $ \a'<\a$, to a $C^{2,\a}$-function $v\colon \Sigma\times (-\infty, T)\to\R$, which satisfies the mean curvature flow equation \eqref{gmcf2}, 
\[
w^-(x,t)\le v(x, t)\le w^+(x,t)\quad \forall t\in (-\infty, T)\,,
\]
and
\[
\|v\|_{C^{2,\a}(\Sigma\times(t-1,t))}\le Ce^{-\l_1 t}\quad \forall t\in (-\infty, T)\,.
\]
The time level sets of this function give an ancient solution with the required estimates.

Finally, we need to show that these solutions are mean convex. We assume again that $\phi_1>0$, with the argument for the negative sign being similar. Note that if the Ricci curvature is positive, then this is immediate by the maximum principle applied to the evolution of $H$. In the general case we can argue as follows. 
Since $w$ satisfies \eqref{wt}, $v$ satisfies the graphical mean curvature flow equation \eqref{gmcf2}, and they both have free boundary, we have
\[
\begin{cases}
(w-v)_t= J_\Sigma(w-v) +E(w)- E(v) + \rho u+u_t & \text{on $\Sigma$}\\
b(w-v)=\e(w)-\e(v) & \text{on $\partial\Sigma$}\,.
\end{cases}
\]
The Schauder estimates \cref{Schauder}, together with the estimates of \cref{E-estimates} and those on $u_t$ in \cref{prop:SurfKthEigenf}, yield
\[
\begin{split}
\|w-v\|_{C^{2,\a}(\Sigma\times(t-1,t))}&\le  C\|w-v\|_{L^2(\Sigma\times(t-2,t))}+ Ce^{2\l t}+{}\\
&\phantom{\le}+C\|w-v\|_{C^{2,\alpha}(\Sigma\times (t-2, t))}(\norm{w}_{C^{2,\alpha}(\Sigma\times (t-2, t))} + \norm{v}_{C^{2,\alpha}(\Sigma\times (t-2, t))}) \,.
\end{split}
\]

Since we have that
$\norm{w}_{C^{2,\alpha}(\Sigma\times (t-2, t))} + \norm{v}_{C^{2,\alpha}(\Sigma\times (t-2, t))}\le Ce^{-\l_1 t}$, for sufficiently small $t$, we can apply \cref{absorption}. Using also the fact that  
\[
\|w(\cdot, t)-v(\cdot,t)\|_{C^0(\Sigma)}\le \|w^+(\cdot, t)-w^-(\cdot,t)\|_{C^0(\Sigma)}\le Ce^{-2\l_1t}
\]
we obtain 
\[
\|w-v\|_{C^{2,\a}(\Sigma\times(t-1,t))}\le Ce^{-2\l_1t}\,.
\]
We can now compare the mean curvature $H_v$ of the graph of $v$, with that of the graph of $w$, $H_w=-\l_1\phi_1e^{-\l_1 t}-\varepsilon u$, as follows
\[
H_v= J_\Sigma(v)+E(v)= H_w-(J_\Sigma(w-v) +E(w)- E(v)+\rho u)
\ge -\l_1\phi_1e^{-\l_1 t}-Ce^{-2\l_1t}\,.
\]
The right hand side is trivially positive for $t$ small enough, which finishes the proof.
\end{proof}

\bibliography{biblio}
\printaddress

\end{document}